\documentclass[a4paper,10pt,intlimits,ngerman]{amsart}
\usepackage[utf8x]{inputenc}
\usepackage[varg]{txfonts}
\usepackage[unicode]{hyperref}
\usepackage{graphicx}
\usepackage{bm}

\makeatletter
\newcommand*{\mint}[1]{%
  \mint@l{#1}{}%
}

\newcommand*{\mint@l}[2]{%
  \@ifnextchar\limits{%
    \mint@l{#1}%
  }{%
    \@ifnextchar\nolimits{%
      \mint@l{#1}%
    }{%
      \@ifnextchar\displaylimits{%
        \mint@l{#1}%
      }{%
        \mint@s{#2}{#1}%
      }%
    }%
  }%
}
\newcommand*{\mint@s}[2]{%
  \@ifnextchar_{%
    \mint@sub{#1}{#2}%
  }{%
    \@ifnextchar^{%
      \mint@sup{#1}{#2}%
    }{%
      \mint@{#1}{#2}{}{}%
    }%
  }%
}
\def\mint@sub#1#2_#3{%
  \@ifnextchar^{%
    \mint@sub@sup{#1}{#2}{#3}%
  }{%
    \mint@{#1}{#2}{#3}{}%
  }%
}
\def\mint@sup#1#2^#3{%
  \@ifnextchar_{%
    \mint@sup@sub{#1}{#2}{#3}%
  }{%
    \mint@{#1}{#2}{}{#3}%
  }%
}
\def\mint@sub@sup#1#2#3^#4{%
  \mint@{#1}{#2}{#3}{#4}%
}
\def\mint@sup@sub#1#2#3_#4{%
  \mint@{#1}{#2}{#4}{#3}%
}
\newcommand*{\mint@}[4]{%
  \mathop{}%
  \mkern-\thinmuskip
  \mathchoice{%
    \mint@@{#1}{#2}{#3}{#4}%
        \displaystyle\textstyle\scriptstyle
  }{%
    \mint@@{#1}{#2}{#3}{#4}%
        \textstyle\scriptstyle\scriptstyle
  }{%
    \mint@@{#1}{#2}{#3}{#4}%
        \scriptstyle\scriptscriptstyle\scriptscriptstyle
  }{%
    \mint@@{#1}{#2}{#3}{#4}%
        \scriptscriptstyle\scriptscriptstyle\scriptscriptstyle
  }%
  \mkern-\thinmuskip
  \int#1%
  \ifx\\#3\\\else_{#3}\fi
  \ifx\\#4\\\else^{#4}\fi  
}
\newcommand*{\mint@@}[7]{%
  \begingroup
    \sbox0{$#5\int\m@th$}%
    \sbox2{$#5\int_{}\m@th$}%
    \dimen2=\wd0 %
    \let\mint@limits=#1\relax
    \ifx\mint@limits\relax
      \sbox4{$#5\int_{\kern1sp}^{\kern1sp}\m@th$}%
      \ifdim\wd4>\wd2 %
        \let\mint@limits=\nolimits
      \else
        \let\mint@limits=\limits
      \fi
    \fi
    \ifx\mint@limits\displaylimits
      \ifx#5\displaystyle
        \let\mint@limits=\limits
      \fi
    \fi
    \ifx\mint@limits\limits
      \sbox0{$#7#3\m@th$}%
      \sbox2{$#7#4\m@th$}%
      \ifdim\wd0>\dimen2 %
        \dimen2=\wd0 %
      \fi
      \ifdim\wd2>\dimen2 %
        \dimen2=\wd2 %
      \fi
    \fi
    \rlap{%
      $#5%
        \vcenter{%
          \hbox to\dimen2{%
            \hss
            $#6{#2}\m@th$%
            \hss
          }%
        }%
      $%
    }%
  \endgroup
}

\newcommand{\mvint}{\mint{-}}

\newcommand{\iv}{B} 


\newtheorem{theorem}{Theorem}[section]
\newtheorem{mtheorem}{Theorem}

\newtheorem{proposition}[theorem]{Proposition}
\newtheorem{lemma}[theorem]{Lemma}

\newtheorem{corollary}[theorem]{Corollary}
\newtheorem{definition}[theorem]{Definition}
\newtheorem{remark}[theorem]{Remark}
\newtheorem{example}[theorem]{Example}
\newtheorem{conjecture}[theorem]{Conjecture}

\hypersetup{
breaklinks=true,
colorlinks=true,
linkcolor=blue,
citecolor=blue,
urlcolor=blue,
}

\numberwithin{equation}{section}

\newcommand{\R}{\mathbb{R}}
\newcommand{\Z}{\mathbb{Z}}
\newcommand{\N}{\mathbb{N}}

\renewcommand{\S}{\mathbb{S}}

\newcommand{\solv}{{\boldsymbol{v}}}
\newcommand{\sol}{{\boldsymbol{\gamma}}}
\newcommand{\solu}{{\boldsymbol{u}}}
\newcommand{\vp}{{\boldsymbol{\varphi}}}

\newcommand{\lap}{\Delta }
\newcommand{\aleq}{\precsim}
\newcommand{\ageq}{\succsim}
\newcommand{\aeq}{\approx}

\newcommand{\laps}[1]{|\nabla|^{#1}}
\newcommand{\lapms}[1]{|\nabla|^{-\brac{#1}}}

\newcommand{\brac}[1]{\left ( #1 \right )}
\newcommand{\abs}[1]{\left | #1 \right |}
\newcommand{\eps}{\varepsilon}

\newcommand{\go}[2]{#1\rhd#2}
\newcommand{\igo}[2]{\oint_{\go{#1}{#2}}}
\newcommand{\migo}[2]{\mvint_{\go{#1}{#2}}}

\newcommand{\dmv}[2]{\langle \delta #1, \delta #2\rangle}

\DeclareMathOperator{\distor}{distor}

\title[O'hara knot energies I]{On O'hara knot energies I: Regularity for critical knots}
\author{Simon Blatt}
\address[Simon Blatt]{Paris Lodron Universit\"at Salzburg, Hellbrunner Strasse 34, 5020 Salzburg, Austria}
\email{simon.blatt@sbg.ac.at}

\author{Philipp Reiter}
\address[Philipp Reiter]
{Fakult\"at f\"ur Mathematik, Universit\"at Duisburg-Essen, Forsthausweg 2, 47057 Duisburg, Germany} \email{philipp.reiter@uni-due.de}

\author{Armin Schikorra}
\address[Armin Schikorra]{Department of Mathematics, University of Pittsburgh, 301 Thackeray Hall, Pittsburgh, PA 15260, USA}
\email{armin@pitt.edu}

\begin{document}


\begin{abstract}
We develop a regularity theory for extremal knots of scale invariant knot energies defined by J.~O'hara in 1991. This class contains as a special case the M\"obius energy.  

For the M\"obius energy, due to the celebrated work of Freedman, He, and Wang, we have a relatively good understanding. Their approch is crucially based on the invariance of the M\"obius energy under M\"obius transforms, which fails for all the other O'hara energies. 

We overcome this difficulty by re-interpreting the scale invariant O'hara knot energies as a nonlinear, nonlocal $L^p$-energy acting on the unit tangent of the knot parametrization. This allows us to draw a connection to the theory of (fractional) harmonic maps into spheres. Using this connection we are able to adapt the regularity theory for degenerate fractional harmonic maps in the critical dimension to prove regularity for minimizers and critical knots of the scale-invariant O'hara knot energies.
\end{abstract}

\maketitle
\tableofcontents

\section{Introduction}
For $\alpha p \geq 4$, $p \geq 2$, the O'hara knot energies\footnote{In this paper we divert from the usual notation of the literature by writing $\frac{p}{2}$ instead of $p$ for the simple reason that the case $p=2$ (in the knot literature: $p=1$) corresponds to the Hilbert-space case of an $L^2$-energy.} $\mathcal{O}^{\alpha,p}$ named after their inventor J.~O'hara \cite{OH91,OH92} are defined for closed Lipschitz curves $\sol: \R/\Z \to \R^3$ from the circle $\R/\Z$ into the three-dimensional space $\R^3$ by
\begin{equation}\label{eq:oharaenergiesnew}
 \mathcal{O}^{\alpha,p}(\sol)= \int_{\R/\Z} \int_{\R/\Z} \brac{\frac{1}{|\sol(x) - \sol(y)|^\alpha} - \frac 1 {\mathcal{D}_\sol(x,y)^\alpha }}^{\frac{p}{2}} |\sol'(x)|\, |\sol'(y) |\ dx\ dy.
\end{equation}
Here, $\mathcal{D}_\sol(x,y)$ denotes the intrinsic distance between $\sol(x)$ and $\sol(y)$ along the curve $\sol$. 

A striking property of an O'hara knot energy is its self-repulsive effect which is in contrast to other geometric energies like e.g. Bernoulli's bending energy. Namely, curves with self-intersections have infinite energy. 

The original definition of the O'hara knot energies in \cite{OH91,OH92} was motivated by the idea of calculating the potential energy of electric charge equidistributed on the given curve. He adapted ideas in the work of Fukuhara \cite{Fukuhara1988} who attempted to use energies to find especially nice representatives of a given knot class.  The subject has found applications in other parts of mathematics and the sciences. Knot energies can help to model repulsive forces of fibres, whenever self-interaction of strands should be avoided. In fact, there is some indication that DNA molecules seek to attain a minimum state of a suitable energy~\cite{Moffatt1996}. There have been several attempts to employ self-avoiding energies for mathematical models in microbiology, see e.g. \cite{Banavar2003}  and references therein for links to polymer science and protein science. Moreover, attraction phenomena may be modeled by a corresponding positive gradient flow \cite{Reiter2009}. Self-repulsive terms can also be used to penalize self-penetration of physical objects, e.g. in elasticity theory \cite{K18}.

From the analytical point of view, one important driving force in this field is the hope that knot energies might lead to a very natural proof of the Smale conjecture. In one of its many equivalent forms this conjecture states that the space of circles is a deformation retract of all Jordan curves. This conjecture was proven by Hatcher \cite{Hatcher1983}. If it could be shown that, in the unknot class, circles are the only critical knots of one of the O'hara energies this could be used to construct a deformation retract proving Smale's conjecture in a straightforward way.

Of course, the central analytic questions for a variational energy are the questions of existence of minimizers, their regularity, and - if possible - even a list or classification of all minimizers. There are two different regimes, analytically: if $\alpha p > 4$ the energy is sensitive to scaling $\sol \mapsto \lambda \sol$, i.e. the size of the knot. Fixing the length of the curve (to ensure existence of minimizers) one finds that the energy is \emph{sub}-critical, e.g. in the sense that sequences of knots with bounded energy (up to a subsequence) converge in the ambient isotopy class. 

For $\alpha p = 4$ the energy $\mathcal{O}^{\alpha,p}$ is invariant under scalings $\sol \mapsto \lambda \sol$. In this case the energy $\mathcal{O}^{\alpha,p}$ measures only the physical shape and does not take into account the size of the knot $\sol$. This is the critical (and thus analytically most challenging) case, e.g. in the sense that sequences of knots with bounded energy may exhibit bubbling phenomena.

Let us also remark that in the (generally in this paper excluded) case $\alpha p < 4$ the energy is not repulsive any more, i.e. self-intersections are permissible.

In this work we will concentrate on the critical case where $\alpha p = 4$. We draw our motivation from the results known in the special case $\alpha = 2$ and $p = 2$, in which the O'hara energy $\mathcal{O}^{2,2}$ is called M\"obius energy. Namely, the following was shown in the groundbreaking work by Freedman, He, and Wang \cite{Freedman1994}.
\begin{theorem}[Freedman-He-Wang]\label{th:FHW}${}$
\begin{enumerate}
\item Minimizers of the M\"obius energy $\mathcal{O}^{2,2}$ exist within any prime knot class.
\item Any local minimizer of $\mathcal{O}^{2,2}$ is of class $C^{1,1}$ when parametrized by arclength.
\end{enumerate}
\end{theorem}
Here, a knot class is the ambient isotopy class of an embedded curve. A knot is called a \emph{prime knot} if it cannot be decomposed into two non-trivial knots. Otherwise it is called a \emph{composite knot}, cf. \cite{BZH}.
Kusner and Sullivan conjectured that there are no minimizers in composite knot classes, supporting their intuition with numerical experiments \cite{Kusner1998}. It is an interesting fact  that the existence in prime knot classes corresponds to what is known for minimizing harmonic maps \cite{SU81} where minimizers exist in generators of the homotopy group and may not exist in other elements of the homotopy group.

The main result of this work is the following regularity theorem extending Freedman, He, Wang's result to the case $\alpha p = 4$. The question of existence will be treated in the forthcoming paper \cite{BRS22}.
\begin{mtheorem}\label{th:regularity}
If $\alpha p = 4$, $p\geq 2$, then any critical knot $\sol$ of $\mathcal{O}^{\alpha,p}$ parametrized by constant speed is of class $C^{1,\sigma}$ for some $\sigma > 0$. In particular this holds for local minimizers.
\end{mtheorem}

Before we explain the main ideas behind the proof of Theorem~\ref{th:regularity} let us gather further facts about O'hara energies.
\subsection*{Known results for O'hara energies}
For $\alpha p \geq 4$, $p \geq 2$, circles are the global minimizers of all O'hara energies among curves of fixed length \cite{Abrams2003}. Regarding minimizers in a given knot class: for the scaling-dependent case $\alpha p > 4$, O'hara \cite{OH91,OH92} proved the existence of minimizers in any knot class. His argument relies on the fact that for $\alpha p > 4$ the knot energies are coercive on the space of embedded $C^{1, \beta}$ curves and lower semi-continuous with respect to $C^1$ convergence. This can nowadays be derived from the classification of curves of finite energy in \cite{Blatt12}: curves with a uniform bound on $\mathcal O^{\alpha,p}$ are uniformly bounded in the fractional Sobolev space $W^{1+\frac{\alpha}{2}-\frac{1}{p},p}$ which continuously embeds into $C^{1+\beta}$ for $\beta = \frac \alpha 2 -  \frac 2 p >0$, if $\alpha p > 4$. Observe that this embedding fails if $\alpha p = 4$.

Adapting arguments of O'hara and Freedman, He, and Wang, the existence of symmetric but non-minimizing critical points was proven in  \cite{Gv18,BGRv19} for both, the non scale-invariant case $\alpha p > 4$ and the M\"obius energy $(\alpha,p) = (2,2)$ (in that case only in prime knot classes). This extends previous work of Kim and Kusner on torus knots \cite{Kim1993} and of Cantarella, Fu, Mastin, Royal on symmetric critical knots for the rope-length \cite{Cantarella2014}.

For $\alpha > 2$, $p=2$ it was proven in \cite{BR12} that critical knots are $C^\infty$; analyticity of solutions was shown in \cite{V19}. Furthermore, $C^\infty$-smoothness of critical knots of the M\"obius energy $\mathcal{O}^{2,2}$ is due to He \cite{He00}, under a $C^{1,1}$ initial regularity assumption, and under only finite energy assumption this was obtained in \cite{BRS16}. Analyticity in this case was shown in \cite{BV19}. 

In a series of papers Ishizeki and Nagasawa \cite{IshizekiNagasawa, IshizekiNagasawaII,Ishizeki16, IshizekiNagasawaIII} showed and analyzed various decomposition results for the M\"obius energy and more generally O'hara energies, see also~\cite{N18}.

\subsection*{Main ideas behind Theorem~\ref{th:regularity}}
A first na\"ive attempt of proving Theorem~\ref{th:regularity} would be to try to extend the geometric arguments in \cite{Freedman1994}. This attempt is thwarted, however, because their proof crucially exploits the M\"obius invariance of the energy $\mathcal{O}^{2,2}$. But this means that there is no hope of extending these arguments: indeed in Section~\ref{s:distortion} we present strong numerical evidence for the following conjecture.

\begin{conjecture}\label{conj}
The energy $\mathcal{O}^{\alpha,p}$ is M\"obius invariant if and only if $(\alpha,p) = (2,2)$.
\end{conjecture}
The underlying strategy behind the proof of Theorem~\ref{th:regularity} is thus completely different from \cite{Freedman1994}. We relate critical knots of $\mathcal{O}^{\alpha,p}$ to harmonic maps. To achieve this we define a new energy $\mathcal{E}^{\alpha,p}$ in \eqref{eq:E}  with the following property: if a constant-speed parametrized knot $\sol$ is a critical knot of $\mathcal{O}^{\alpha,p}$, then the unit 
tangent  $\solu:= \frac {\sol'}{|\sol'|}$  is a critical map of the energy
$
 \mathcal{E}^{\alpha,p} 
$
restricted to maps into the unit sphere $\S^2$, $\solv: \mathbb R / \mathbb Z \rightarrow \mathbb S^2$, cf. Theorem~\ref{th:eregularity}. 

The energy $\mathcal{E}^{\alpha,p}$ has a nonlinear, nonlocal Lagrangian; and $\mathcal{E}^{\alpha,p}$ is comparable to a $W^{\frac{1}{p},p}(\R/\Z)$-seminorm, where $W^{\frac{1}{p},p}$ denotes the fractional Sobolev space. Critical maps of the $W^{\frac{1}{p},p}$-seminorm are called $W^{\frac{1}{p},p}$-harmonic maps. As a consequence, the described reformulation allows us to interpret critical knots $\sol$ of $\mathcal{O}^{\alpha,p}$-energies as (essentially) fractional harmonic-type maps $\solu$ into the sphere $\S^2$.

Fractional harmonic maps have been studied first by Da Lio and Rivi\`{e}re \cite{DR1dSphere,DR1dMan} in the form of $W^{\frac{1}{2},2}$-harmonic maps. These maps are in turn generalizations of classical harmonic maps, as treated e.g. on two-dimensional domains by H\'elein \cite{Helein-1990}. This theory was extended to $W^{1/p,p}$-harmonic maps into spheres by the third-named author in \cite{SchikorraCPDE14}, see also \cite{MS18}. We will further extend this theory to the case of $\mathcal{E}^{\alpha,p}$-harmonic maps into spheres, see Theorem~\ref{th:eregularity}.  In view of our identification of critical knots $\sol$ for the O'hara energy $\mathcal{O}^{\alpha,p}$ with a critical map $\solu=\frac{\sol'}{|\sol'|}$ with respect to the energy $\mathcal{E}^{\alpha,p}$ we obtain Theorem~\ref{th:regularity}.

Let us remark, that the strategy of interpreting $\sol'$ as a solution to a fractional harmonic-type map equation has already been proven successful in \cite{BRS16} for the case of the M\"obius energy $(\alpha,p) = (2,2)$ -- i.e. the Hilbert-space case. 

\begin{remark}
A minor technical adaptation of the arguments in this article leads  to the corresponding statement of Theorem~\ref{th:regularity} for so-called open knots, $\sol: \R \to \R^3$, which are critical points of the energy
\begin{equation}\label{eq:oharaenergiesnewo}
 \mathcal{O}^{\alpha,p}_o(\sol)= \int_{\R} \int_{\R} \brac{\frac{1}{|\sol(x) - \sol(y)|^\alpha} - \frac 1 {\mathcal{D}_\sol(x,y)^\alpha }}^{\frac{p}{2}} |\sol'(x)|\, |\sol'(y) |\ dx\ dy.
\end{equation}
\end{remark}

\subsection*{Notation}
We will use fairly standard notation: $\iv$ will denote intervals in $\R$. 
In estimates we write $a \aleq b$ if $a \leq C\, b$ for some multiplicative constant $C$. We write $a \ageq b$ if $b\aleq a$, and $a \aeq b$ if $a \aleq b$ and $b \aleq a$.

We work with two types of critical points of energy functionals, which for the convenience of the reader we label differently: we will denote as \emph{critical knots} the critical (i.e. stationary) points $\sol: \R / \Z \to \R^3$ of the O'hara energy $\mathcal{O}^{\alpha,p}$. As \emph{critical maps} we will denote the critical points $\solu: \R / \Z \to \S^2$ of the (yet to be defined) energy functional $\mathcal{E}^{\alpha,p}$.

\subsection*{Acknowledgments.}

Financial support is acknowledged as follows
\begin{itemize}
\item FWF Austrian Science Fund, grant no P 29487-N32 (SB)
\item German Research Foundation (DFG) through grant no.~RE-3930/1-1 (PR)
\item German Research Foundation (DFG) through grant no.~SCHI-1257/3-1 (AS)
\item Simons foundation through grant no.~579261 (AS)
\item Daimler and Benz foundation, grant no 32-11/16 (AS) 
\end{itemize}
Part of this work was carried out while A.S. was a Heisenberg fellow.

\section{A new energy \texorpdfstring{$\mathcal{E}^{\alpha,p}$}{E}}
In this section we find a new energy $\mathcal{E}^{\alpha,p}(\solu)$ applied to maps $\solu: \R/\Z \to \S^2$ so that any critical knot (parametrized by arc-length) induces a critical $\S^2$-valued map $\solu$ and vice versa. 

The main result in this section is the following equivalence between an energy $\mathcal{E}^{\alpha,p}$ which will be defined below and the O'hara energy $\mathcal{O}^{\alpha,p}$.

\begin{theorem}\label{th:criticalpointchar}
For $\alpha p = 4$, $p \geq 2$. Let $\sol: \R/\Z \to \R^3$ be a knot parametrized with constant speed with finite energy $\mathcal{O}^{\alpha,p}(\sol) < \infty$. 

Denote with $\solu := c^{-1} \sol': \R/\Z \to \S^2$ its unit tangent field, where $c \equiv |\sol'|$.
Then $(1) \Rightarrow (2)$, where
\begin{enumerate}
 \item $\sol$ is a critical knot for $\mathcal{O}^{\alpha,p}$, that is $\mathcal{O}^{\alpha,p}(\sol) < \infty$ and for any $\varphi \in C^\infty(\R/\Z,\R^3)$ we have
\[
 \delta\mathcal{O}^{\alpha,p}(\sol,\varphi) := \frac{d}{d\eps} \Big |_{\eps = 0} \mathcal{O}^{\alpha,p}(\sol+\eps \varphi) = 0.
\]
 \item $\solu$ is a critical map for $\mathcal{E}^{\alpha,p}$ in the class of maps $\solv: \R/\Z \to \S^2$ with finite energy $\mathcal{E}^{\alpha,p}$, that is for any $\varphi \in C^\infty(\R/\Z,\R^3)$ we have
 \[
  \frac{d}{d\eps} \Big |_{\eps = 0} \mathcal{E}^{\alpha,p}\brac{\frac{\solu+\eps \varphi}{|\solu+\eps \varphi|}} = 0.
 \]

\end{enumerate}
\end{theorem}
Let us remark again that for the convenience of the reader we will refer to critical \emph{knots} when referring to criticality with respect to the O'hara energy $\mathcal{O}^{\alpha,p}$ acting on in isotopy classes, and a critical \emph{map} when we refer to criticality with respect to the energy $\mathcal{E}^{\alpha,p}$ as maps between the manifolds $\R/\Z \to \S^2$. In principle both are of course critical points of the respective energy in a certain domain.

The important advantage of $u$ being a critical map in the above sense for $\mathcal{E}^{\alpha,p}$ is that the class of permissible is independent of the topological constraints of $\sol$, and are exactly the variations permitted for harmonic maps into $\S^2$ -- thus the regularity theory for harmonic maps comes into play.

To define properly the energy $\mathcal{E}^{\alpha,p}$ we first fix some notation for integrating on $\R/\Z$. The reader may decide to skip the following notation and use common sense when interpreting the meaning of respecitve integrals.
\begin{remark}[Technicalities on integration on segments of $\R/\Z$]
\begin{enumerate} 
\item We denote with $\rho(x,y)$ the distance on $\R/\Z$, namely $\rho(x,y) = |x-y| \mod 1$.
 \item For $x,y \in \R/\Z$ there are two geodesics. If $x$ and $y$ are not not antipodal (i.e. $|x-y| \neq \frac{1}{2}$) let $\go{x}{y}$ be the shortest geodesic (with orientation $x$ to $y$), and for this case we define the integral, for $\Z$-periodic $f$,
 \[
  \igo{x}{y} f\ ds:= \int_{x}^{\tilde{y}} f(z)\, dz(z),
\]
where $\tilde{y} \in y + \Z$ such that $|x-\tilde{y}| < \frac{1}{2}$.
Note that this is an integral with orientation
in the sense that
\[
  \igo{x}{y} f\ ds  =  -\igo{y}{x} f\ ds 
\]
\item We define the mean value integral on a segment $S \subset \R/\Z$
\[
 \mvint_{\iv} f := \frac{1}{\mathcal{H}^1(S)} \int_S f\ d\mathcal{H}^1.
\]
(i.e. the mean value is always without taking into account the orientation). In particular,
\[
 (f)_{\R/\Z} =  \int_{\R/\Z} f = \int_{0}^{1} f(z)\ dz.
\]
Moreover we write
\[
 \migo{x}{y} f \equiv \mvint_{|\go{x}{y}|} f = \frac{\sigma(\go{x}{y})}{\rho(x,y)}\ \igo{x}{y} f\ ds,
\]
where
\[
 \sigma(\go{x}{y}) = {\rm sgn} \igo{x}{y} 1,
\]
that is $\sigma(\go{x}{y})$ is $+1$ if $\go{x}{y}$ from $x$ to $y$ is positively oriented, and $-1$ if it is negatively oriented. In particular,
\[
 \migo{x}{y} f= \migo{y}{x} f,
\]
and
\[
 \migo{x}{y} f \geq 0 \quad \mbox{if $f \geq 0$ on $\go{x}{y}$}.
\]
\item We then have the fundamental theorem of calculus in the following form. For any non-antipodal points $x,y \in \R/\Z$ we have
\[
 f(y)-f(x) = \igo{x}{y} f'(z)\ dz,
\]
where $f'$ is the derivative.
\end{enumerate}
\end{remark}
With this notation, we can define the energy $\mathcal{E}^{\alpha,p}$. Firstly for two maps $\solu,\solv: \R/\Z \to \R^3$ we denote 
\begin{equation}\label{eq:bracketdef}
 \dmv{\solu}{\solv}(x,y) :=  \migo{x}{y}\migo{x}{y} \big (\solu(z_1) - \solu(z_2) \big ) \cdot \big (\solv(z_1) - \solv(z_2) \big )\ dz_1\ dz_2.
 \end{equation}
For nonnegative $a,b,c$ we set
\begin{equation}\label{eq:Flagrangian}
 F(a,b,c) := \brac{\brac{b- \frac{1}{2} a}^{-\frac{\alpha}{2}}   -c^{-\alpha}}^{\frac{p}{2}}.
\end{equation}
Also for later use define
\begin{equation}\label{eq:Glagrangian}
 G(a) := F(a,1,1) = \brac{\brac{1- \frac{1}{2} a}^{-\frac{\alpha}{2}}  -1}^{\frac{p}{2}}.
\end{equation}
Now the energy $\mathcal{E}^{\alpha,p}$ is defined as follows.
\begin{equation}\label{eq:E}
 \mathcal{E}^{\alpha,p}(\solu) = \int_{\R/\Z}\int_{\R/\Z} F\brac{\dmv{\solu}{\solu}(x,y), \migo{x}{y} |\solu-(\solu)_{\R/\Z}|^2 , \migo{x}{y} |\solu-(\solu)_{\R/\Z}|}\ |\solu(x)- (\solu)_{\R/\Z}|\, |\solu(y)-(\solu)_{\R/\Z}|\ \frac{dx\ dy}{\rho(x,y)^{\alpha p /2}}.
\end{equation}
We will see later, see Proposition~\ref{la:EOequal}, that $\mathcal{E}^{\alpha,p}$ is comparable to the Gagliardo-norm $[\solu]_{W^{\frac{\alpha}{2}-\frac{1}{p},p}}^p$, so we are formally in the realm of $W^{\beta,p}$-harmonic maps into the sphere, for which regularity theory has been developed by the third-named author \cite{SchikorraCPDE14}, see also \cite{MS18}.

\subsection{Properties of \texorpdfstring{$\mathcal{E}^{\alpha,p}$}{Eap}}
Let $\sol: \R/\Z \to \R^3$ be a constant speed parametrization of a knot. For $c := |\sol'|$ we set $u := c^{-1}  \sol'$. 

In \cite{Blatt12} the first-named author characterized the energy space of knots with finite O'hara energies. Which imply the following properties, Lemma~\ref{la:lambda} and Lemma~\ref{la:Gpest}.

\begin{lemma}\label{la:lambda}
Whenever $\sol$ has finite energy $\mathcal{O}^{\alpha,p}$, $\alpha p = 4$, $p \geq 2$, then there exists $\lambda = \lambda(\solu)$ such that
\begin{equation}\label{eq:ubilipschitz}
 \sup_{x,y \in \R / \Z} \dmv{\solu}{\solu}(x,y) \leq \lambda.
\end{equation}
 \end{lemma}
\begin{proof}
In \cite[Lemma 2.1]{Blatt12} the first-named author showed that whenever $\sol$ has finite energy $\mathcal{O}^{\alpha,p}$ then there exists a bilipschitz constant $L = L(\sol) > 0$ such that
\[
 L \leq \frac{\sol(x)-\sol(y)}{|x-y|}.
\]
Consequently, since $\solu = c^{-1} \sol'$,
\[
\begin{split}
 \dmv{\solu}{\solu}(x,y)= &\igo{x}{y} \igo{x}{y} |\solu(s)-\solu(t)|^2\, ds\, dt\\
 =& c^{-2} \igo{x}{y} \igo{x}{y} |\sol'(s)-\sol'(t)|^2\, ds\, dt\\
 =& \igo{x}{y} \igo{x}{y} 2-\frac{2}{c^2} \langle \sol'(s), \sol'(t)\rangle \, ds\, dt\\
=&\brac{ 2-2 c^{-2}\frac{|\sol(y)-\sol(x)|^2}{|x-y|^2}}\\
\leq& 2-2 c^{-2}L^2 =: \lambda.
\end{split}
 \]
\end{proof}
\begin{lemma}\label{la:Gpest}
For $s \in [0,\lambda]$ for some $\lambda < 2$,
\[
 |G(s)| \aeq s^{\frac{p}{2}}.
\]
the derivative
\[
 |G'(s)| \aeq s^{\frac{p-2}{2}}.
\]
The constants depend on $\lambda$.
\end{lemma}
\begin{proof}
Firstly, for $s \in [0,\lambda]$ with L'hopital one obtains
\[
 \brac{1- \frac{1}{2} s}^{-\frac{\alpha}{2}}  -1 \aeq s.
\]
Also for all $s \in [0,\lambda]$,
\[
 (1- \frac{1}{2} s) \aeq 1.
\]
Finally compute and conclude with
\[
 G'(s) := \frac{\alpha p}{8} \brac{\brac{1- \frac{1}{2} s}^{-\frac{\alpha}{2}}  -1}^{\frac{p-2}{2}}\ \brac{1- \frac{1}{2} s}^{-\frac{\alpha+2}{2}}.
\]
\end{proof}
For $\solu: \R/\Z \to \S^2$ satisfying \eqref{eq:ubilipschitz} and $\int_{\R/\Z} \solu = 0$, the energy $\mathcal{E}^{\alpha,p}(\solu)$ is comparable to the following $W^{\frac{\alpha}{2}-\frac{1}{p},p}$-seminorm $\llbracket f \rrbracket_{W^{\beta,p}(\R)}$ which for an interval $\iv$ is defined as
\[
 \llbracket f \rrbracket_{W^{\beta,p}(\iv)} = \brac{\int_{\iv}\int_{\iv} \frac{\dmv{f}{f}(x,y)^{\frac{p}{2}}}{\rho(x,y)^{1+\beta p}}\ dx\ dy}^{\frac{1}{p}}.
\]
The notion $W^{\beta,p}(\iv)$ is justified, since for $\beta$ sufficiently large in terms of $p$ we have the following equivalence to the usual $W^{\beta,p}$-norm,
\[
 [f]_{W^{\beta,p}(\iv)} = \brac{\int_{\iv}\int_{\iv} \frac{|f(x)-f(y)|^p}{\rho(x,y)^{1+\beta p}}\ dx\ dy}^{\frac{1}{p}}.
\]
For
\begin{proposition}\label{pr:Wspequivalence}
For $p \in (1,\infty)$, $\beta \in (0,1)$ such that $1 > \beta > \frac{1}{p}-\frac{1}{2}$ we have (with constants depeding on $p$, $\beta$)
\[
 \llbracket \solu \rrbracket_{W^{\beta,p}(\R/\Z)} \aeq [\solu]_{W^{\beta,p}(\R/\Z)}.
 \]
\end{proposition}
The proof is given in the appendix, see Proposition~\ref{pr:normalunnormalWsp}. 

\subsection{Proof of Theorem~\ref{th:criticalpointchar}}
We introduce an auxiliary energy $\tilde{\mathcal{E}}^{\alpha,p}$ which is the same as $\mathcal{E}^{\alpha,p}$ only that the third component in $F$, namely $\migo{x}{y} |\solu-(\solu)_{\R/\Z}|$, is replaced by 
\[
 d_{\solu}(x,y) := \min \left \{ \rho(x,y)\migo{x}{y} |\solu(z)-(\solu)_{\R/\Z}|, (1-\rho(x,y)) \mvint_{\R/\Z \backslash \go{x}{y}}  |\solu(z)-(\solu)_{\R/\Z}| \right \}.
\]
Observe that when $\solu = \sol'$ then $d_{\solu}(x,y)$ simplifies to
\[
 d_{\solu}(x,y) \equiv \mathcal{D}_{\sol}(x,y):= \min \left \{ \rho(x,y)\migo{x}{y} |\solu(z)|, (1-\rho(x,y)) \mvint_{\R/\Z \backslash \go{x}{y}} |\solu(z)| \right \},
\]
i.e. the intrinsic distance of $\sol(x)$ to $\sol(y)$ along $\sol$.

That is,
\[
\tilde{\mathcal{E}}^{\alpha,p}(\solu) = \int_{\R/\Z}\int_{\R/\Z} F\brac{\dmv{\solu}{\solu}(x,y), \migo{x}{y} |\solu-(\solu)_{\R/\Z}|^2 , \frac{d_{\solu}(x,y)}{\rho(x,y)}}\ |\solu(x)- (\solu)_{\R/\Z}|\, |\solu(y)-(\solu)_{\R/\Z}|\ \frac{dx\ dy}{\rho(x,y)^{\alpha p /2}}.
\]
Then we have
\begin{lemma} \label{la:EOequal}
For any knot $\sol$ with finite energy $\mathcal{O}^{\alpha,p}$ we have
\[\tilde{\mathcal{E}}^{\alpha,p}(\sol') = \mathcal{O}^{\alpha,p}(\sol).\] 
If the knot $\sol$ is moreover \emph{arclength-parametrized} or more generally \emph{constant speed},
\[\mathcal{E}^{\alpha,p}(\sol') = \mathcal{O}^{\alpha,p}(\sol).\] 
\end{lemma}
\begin{proof}
Firstly, $(\sol')_{\R/\Z} = 0$. Thus $\mathcal{D}_{\sol}(x,y) = d_{\sol'}(x,y)$, and we have
\[
\tilde{\mathcal{E}}^{\alpha,p}(\sol') = \int_{\R/\Z}\int_{\R/\Z} F\brac{\dmv{\sol'}{\sol'}(x,y), \migo{x}{y} |\sol'|^2 , \frac{\mathcal{D}_{\sol}(x,y)}{\rho(x,y)}}\ |\sol'(x)|\, |\sol'(y)|\ \frac{dx\ dy}{\rho(x,y)^{\alpha p /2}}.
\]
We have
\[
 F\brac{\dmv{\sol'}{\sol'}(x,y), \migo{x}{y} |\sol'|^2 , \frac{\mathcal{D}_{\sol}(x,y)}{\rho(x,y)}}
 =\brac{\brac{\migo{x}{y}|\sol'|^2- \frac{1}{2}\dmv{\sol'}{\sol'}}^{-\frac{\alpha}{2}}-\brac{\frac{\rho(x,y)}{\mathcal{D}_{\sol}(x,y)}}^{\alpha}}^{\frac{p}{2}}
 \]
We recall $\dmv{\sol'}{\sol'}$ defined in \eqref{eq:bracketdef}. Then we have (the orientation cancels),
\[
 \frac{|\sol(x)-\sol(y)|^2}{\rho(x,y)^2} =  \migo{x}{y} |\sol'|^2-\frac{1}{2}\migo{x}{y}\migo{x}{y} |\sol'(s)-\sol'(t)|^2  = \migo{x}{y} |\sol'|^2 - \frac{1}{2} \dmv{\sol'}{\sol'}(x,y).
 \]
Thus,
\[
\frac{F\brac{\dmv{\sol'}{\sol'}(x,y), \migo{x}{y} |\sol'|^2 , \mathcal{D}_{\sol}(x,y)}}{\rho(x,y)^{\frac{\alpha p}{2}}}\\
 =\brac{\frac{1}{|\sol(x)-\sol(y)|^{\alpha}}-\frac{1}{\mathcal{D}_{\sol}(x,y)^{\alpha}}}^{\frac{p}{2}}\\
 \]
This shows $\tilde{\mathcal{E}}^{\alpha,p}(\sol') = \tilde{O}^{\alpha,p}(\sol')$.

If $\sol$ is arclength-parametrized,
\[
 \frac{\mathcal{D}_{\sol}(x,y)}{\rho(x,y)} = 1 = \migo{x}{y} |\solu-(\solu)_{\R/\Z}|.
\]
Thus, $\mathcal{E}^{\alpha,p}(\sol') = \tilde{\mathcal{E}}^{\alpha,p}(\sol')$ for arclength parametrized $\sol$.
\end{proof}

Now we are ready to show

\begin{proof}[Proof of Theorem~\ref{th:criticalpointchar}: (1) $\Rightarrow$ (2)]
Observe for any $c > 0$,
\[
 \mathcal{E}^{\alpha,p}(c^{-1}\sol') = c^{\alpha p-2}\mathcal{E}^{\alpha,p}(\sol'),
\]
and
\[
 \tilde{\mathcal{E}}^{\alpha,p}(c^{-1}\sol') = c^{\alpha p-2}\tilde{\mathcal{E}}^{\alpha,p}(\sol').
\]
Thus, we shall assume w.l.o.g. $\sol' = u$ (i.e. $c \equiv 1$) without changing anything about the criticality.

For $\vp \in C^\infty(\R/\Z ,\R^3)$ let
\[
 \solu_\eps := \frac{\solu+\eps \vp}{|\solu+\eps \vp|},
\]
and 
\[
 \sol_\eps(x) := \sol(0) + \igo{0}{x} \brac{\solu_\eps -(\solu_\eps)_{\R/\Z}}.
 \]
Then, since $\solu_\eps-\sol_\eps'   = (\solu_\eps)_{\R/\Z}$ is constant, we find
\[
 \tilde{\mathcal{E}}^{\alpha,p}(\solu_\eps) = \tilde{\mathcal{E}}^{\alpha,p}(\solu_\eps-(\solu_\eps)_{\R/\Z})= \tilde{\mathcal{E}}^{\alpha,p}(\sol_\eps').
\]
On the other hand, by Lemma~\ref{la:EOequal}, for all small $\eps$ (so that the energies are finite),
\[
 \mathcal{O}^{\alpha,p}(\sol_\eps) = \tilde{\mathcal{E}}^{\alpha,p}(\sol_\eps').
\]
Moreover, we observe that
\begin{equation}\label{eq:soluepsexpansion}
 \solu_\eps  = \solu + \eps \Pi(\solu) \vp  + O(\eps^2),f
\end{equation}
where $\Pi(\solu)$ is the orthogonal projection onto $T_\solu \S^2$,
\[
 \Pi(\solu)\solv = \solv-\langle \solv,\solu\rangle \solu
\]
and
\[
 \sol_\eps(x) = \sol(x) + \eps\igo{0}{x} \brac{\Pi(\solu)\vp-(\Pi(\solu)\vp)_{\R/\Z}} + O(\eps^2).
 \]
Thus $\sol_\eps$ is a variation of $\sol$, so because of $\sol$ being by (i) a critical knot of $\mathcal{O}^{\alpha,p}$ we have
\begin{equation}
\label{eq:OtE}
\frac{d}{d\eps}\Big |_{\eps = 0} \tilde{\mathcal{E}}^{\alpha,p}(\solu_\eps) = \frac{d}{d\eps}\Big |_{\eps = 0}\mathcal{O}^{\alpha,p}(\sol_\eps) = 0.
\end{equation}

In order to make the step from $\tilde{\mathcal{E}}$ to $\mathcal{E}$, recall that the $\tilde{\mathcal{E}}$ is $\mathcal{E}$ only replacing $\frac{d_{\solu_\eps}(x,y)}{\rho(x,y)}$ by $\migo{x}{y} |\solu_\eps(z)-(\solu_\eps)_{\R/\Z}|$. We show that this does not change the derivative $\frac{d}{d\eps} \Big|_{\eps = 0}$, by Lebesgue theorem:

Firstly, by definition
\[
 \frac{d_{\solu_\eps}(x,y)}{\rho(x,y)} = \min \left \{ \migo{x}{y} |\solu_\eps(z)-(\solu_\eps)_{\R/\Z}|, \frac{(1-\rho(x,y))}{\rho(x,y)} \mvint_{\R/\Z \backslash \go{x}{y}} |\solu_\eps(z)-(\solu_\eps)_{\R/\Z}| \right \}.
\]
Since $|\solu| = 1$ and $(\solu)_{\R/\Z} = 0$, for any $x,y$ not antipodal,
\[
 \migo{x}{y} |\solu_\eps(z)-(\solu_\eps)_{\R/\Z}| = 1+O(\eps).
\]
Also, $\rho(x,y) \leq 1$ and thus
\[
 \frac{(1-\rho(x,y))}{\rho(x,y)} \mvint_{\R/\Z \backslash \go{x}{y}} |\solu_\eps(z)-(\solu_\eps)_{\R/\Z}| = \frac{(1-\rho(x,y))}{\rho(x,y)} + O(\eps) \geq 1 - O(\eps).
\]
Thus we have for any $x,y$ not antipodal.
\[
 1 - O(\eps) \leq \frac{d_{\solu_\eps}(x,y)}{\rho(x,y)} \leq 1 + O(\eps).
\]
Moreover,
\[
 \sup_{\delta \in (0,\eps)} \left |\frac{\frac{d_{\solu_\eps}(x,y)}{\rho(x,y)} - \frac{d_{\solu_\delta}(x,y)}{\rho(x,y)}}{\delta-\eps}  \right | = O(1).
\]
On the other hand, pointwise almost everywhere,
\[
 \frac{d}{d\eps} \Big |_{\eps = 0}\frac{d_{\solu_\eps}(x,y)}{\rho(x,y)} =  \frac{d}{d\eps} \Big |_{\eps = 0} \migo{x}{y} |\solu_\eps(z)-(\solu_\eps)_{\R/\Z}|
\]
and so by the Lebesgue convergence theorem we conclude
\[
 \frac{d}{d\eps} \Big|_{\eps = 0}\mathcal{E}^{\alpha,p}(\solu_\eps) = \frac{d}{d\eps} \Big|_{\eps = 0}\tilde{\mathcal{E}}^{\alpha,p}(\solu_\eps).
\]
With \eqref{eq:OtE} we conclude 
\[
 \frac{d}{d\eps} \Big|_{\eps = 0}\mathcal{E}^{\alpha,p}(\solu_\eps) = 0.
\]
Thus if $\sol$ is an arclength-parametrized critical knot of $\mathcal{O}^{\alpha,p}$, then $c^{-1} \sol'$ is a critical map of $\mathcal{E}^{\alpha,p}$ in the class of maps $v: \R/\Z \to \S^2$.
\end{proof}

\begin{remark}
We found the relation between $\mathcal{E}^{\alpha,p}$ and $\mathcal{O}^{\alpha,p}$ quite intriguing. 

For example: to us it did not seem obious that (2) $\Rightarrow$ (1) in Theorem~\ref{th:criticalpointchar}. By an argument due to He \cite{He2000} any $\mathcal{O}^{\alpha,p}$-variation of $\sol$ (constant-speed parmetrized) into the direction of $\solu=\sol'$ is trivial even if $\sol$ is not a critical knot. Indeed, this is a consequence of the parametrization invariance of $\mathcal{O}^{\alpha,p}$. 
Moroever, $\mathcal{O}^{\alpha,p}$-variations of $\sol$ in direction $\solu \wedge \solu'$ essentially correspond to $\mathcal{E}^{\alpha,p}$-variations of $\solu$ in direction of $T_\solu \S^2$, and thus vanish if $\solu$ is a critical map for $\mathcal{E}^{\alpha,p}$.
However, $\mathcal{O}^{\alpha,p}$-variations of $\sol$ in the direction of $\solu' = \sol''$ seem to correspond neither to reparametrizations of $\sol$ nor to tangential variations of $\mathcal{E}^{\alpha,p}$. However, as our arguments shows, $\mathcal{O}^{\alpha,p}$-criticality with respect to this class of variations is not necessary to obtain the regularity theory for knots -- so in principle one could weaken theorem~\ref{th:regularity} to critical knots with variations in fewer directions.

It also seems not fully clear to us what geometrical implication it is for a map $\solu: \R /\Z \to \S^2$ to have finite energy $\mathcal{E}$, in particular with respect to the (possibly not closed!) corresponding curve $\sol(x) := \int_0^x \solu(z)\, dz$. (e.g. it seems not obvious that $\sol$ is injective).
\end{remark}

\section{Regularity theory for \texorpdfstring{$\mathcal{E}^{\alpha,p}$}{E}-critical maps: Proof of Theorem~\ref{th:regularity}}
In view of Theorem~\ref{th:criticalpointchar} the claim of Theorem~\ref{th:regularity} is a consequence of the following theorem.
\begin{theorem}\label{th:eregularity}
For $\alpha p = 4$, $p \geq 2$ consider $\mathcal{E}^{\alpha,p}$ from \eqref{eq:E}.

Let $\solu$ is a critical map for $\mathcal{E}^{\alpha,p}$ in the class of maps $\solv: \R/\Z \to \S^2$ with finite energy $\mathcal{E}^{\alpha,p}$, that is for any $\varphi \in C^\infty(\R/\Z,\R^3)$ we have
 \[
  \frac{d}{d\eps} \Big |_{\eps = 0} \mathcal{E}^{\alpha,p}\brac{\frac{u+\eps \varphi}{|\solu+\eps \varphi|}} = 0.
 \]
Then $\solu \in C^{\sigma}$ for some $\sigma > 0$.
\end{theorem}
We prove this theorem by extending the argument in \cite{SchikorraCPDE14} where $W^{\beta,p}$-harmonic maps into spheres were considered. See also \cite{MS18} for a different proof. 

\begin{remark}
It is a natural question to ask whether it is possible to extend He's argument, \cite{He2000} see also \cite{BRS16}, which shows smoothness of critical knots of the M\"obius energy that have an initial $C^{1,\sigma}$-regularity. But this argument depends heavily on the $L^2$-type of the energy. From the theory of critical harmonic map-type equations one expects that once $C^{1,\sigma}$-regularity is obtained as in Theorem~\ref{th:regularity} the best possible regularity corresponds to the best possible regularity of solutions to the homogeneous pde of the leading order operator (which in the $L^2$-case is the linear fractional Laplacian $\laps{\beta}$, and thus smoothness is to be expected). In our $L^p$-energy setting the corresponding equation involves however the fractional $p$-Laplace equation for which maximal regularity is unknown. The best current results for that operator has been obtained by Brasco and Lindgren, see \cite{BL15,BLS18}. Regarding convergence, observe that there are suitable stability results available for the fractional $p$-Laplacian \cite{SireKuusiMingioneSelfImproving,S16,ABES19} which seem to carry over to our situation.
\end{remark}

\subsection{Euler-Lagrange equations of \texorpdfstring{$\mathcal{E}^{\alpha,p}$}{Eap}}
To compute the Euler-Lagrange equations of $\mathcal{E}^{\alpha,p}$ we introduce
\[
 \mathcal{Q}(\solu,\varphi) := 2\int_{\R/\Z}\int_{\R/\Z} G'(\dmv{\solu}{\solu}(x,y))\ \dmv{\solu}{\varphi}(x,y)\ \frac{dx\ dy}{\rho(x,y)^{\alpha p /2}},
\]
and
\[
 \mathcal{R}_1(\solu,\varphi) = \int_{\R/\Z}\int_{\R/\Z} H(\dmv{\solu}{\solu}(x,y))\ \migo{x}{y} \solu\cdot (\varphi)_{\R/\Z} \frac{dx\ dy}{\rho(x,y)^{\alpha p /2}}.
\]
with
\[
 H(a) = \frac{\alpha p}{2}\brac{\brac{1- \frac{1}{2} a}^{-\frac{\alpha}{2}}   -1}^{\frac{p-2}{2}} \brac{\brac{1- \frac{1}{2} a}^{-\frac{\alpha+2}{2}}   -1}.
\]

Finally we set 
\[
 \mathcal{R}_2(\solu,\varphi) = \int_{\R/\Z}\int_{\R/\Z} G(\dmv{\solu}{\solu}(x,y)) \brac{\solu(x)+\solu(y) } \cdot (\varphi)_{\R/\Z} \frac{dx\ dy}{\rho(x,y)^{\alpha p /2}}.
\]
Note that in view of Lemma~\ref{la:Gpest}.
\begin{equation}\label{eq:HGest}
H(a), G(a) \aeq a^{\frac{p}{2}}\quad \mbox{for $0 \leq a \leq \lambda < 2$}
\end{equation}

Formally the Euler-Lagrange equations of $\mathcal{E}^{\alpha,p}$ in the class of maps from $\R/\Z$ into $\S^2$ are then 
\begin{equation}\label{eq:formalEL}
 \mathcal{Q}(\solu,\cdot) + \mathcal{R}_1(\solu,\cdot) - \mathcal{R}_2(\solu,\cdot)\perp T_{\solu} \S^2,
\end{equation}

\begin{remark}
We shall see that $\mathcal{R}_i(\solu,\cdot)$ are lower order terms which belong to $(L^1)^\ast$, see Proposition~\ref{pr:loterms}.
Thus, one could interpret \eqref{eq:formalEL} as a fractional, nonlinear version of the harmonic map equation
\[
 \lap \solu \perp T_{\solu} \S^2
\]
up to a lower-order term $\mathcal{R}_i(\solu,\cdot)$.
\end{remark}

More precisely we have the following
\begin{lemma}[Euler-Lagrange equations]\label{la:EL}
Let $\solu$ be a critical point of $\mathcal{E}$ in the class of maps $v: \R/\Z \to \S^2$ and assume $\int_{\R/\Z} \solu = 0$ as well as \eqref{eq:ubilipschitz}. Then 
for any testfunction $\varphi \in W^{\frac{1}{p},p}(\R/\Z, \R^3)$, which is also tangential, $\varphi \in T_{\solu} \S^2$,
\[
 \delta \mathcal{E}^{\alpha,p}(\solu,\varphi) = \mathcal{Q}(\solu,\varphi) + \mathcal{R}_1(\solu,\varphi) - \mathcal{R}_2(\solu,\varphi).
\]
\end{lemma}

\begin{proof}
Of course this only holds almost everywhere, namely whenever $x,y$ are not antipodal. With a cutoff argument we can make the following computations rigorous.

Recall that $G(a) = F(a,1,1)$ where
\[
 F(a,b,c) := \brac{\brac{b- \frac{1}{2} a}^{-\frac{\alpha}{2}}   -c^{-\alpha}}^{\frac{p}{2}}.
\]
Set $\solu_\eps := \solu + \eps \varphi$. We need to compute 
\[
\left. \frac{d}{d\eps} \right |_{\eps = 0} \int_{\R/\Z}\int_{\R/\Z} F\brac{a(\eps), b(\eps) , c(\eps)}\  d(\eps)\, e(\eps) \frac{dx\ dy}{\rho(x,y)^{\frac{\alpha p}{2}}}.
 \]
Here, we set
\[
\begin{split}
 a(\eps) &:= \dmv{\solu_\eps}{\solu_\eps}(x,y)\\
 b(\eps) &:= \migo{x}{y} |\solu_\eps-(\solu_\eps)_{\R/\Z}|^2\\
 c(\eps) &:= \migo{x}{y} |\solu_\eps-(\solu_\eps)_{\R/\Z}|\\
 d(\eps) &:= |\solu_\eps(x)- (\solu_\eps)_{\R/\Z}|\\
e(\eps) &:= |\solu_\eps(y)- (\solu_\eps)_{\R/\Z}|.
 \end{split}
\]
We first observe that $b(0) = c(0) = d(0) = e(0) = 1$ since $(\solu)_{\R/\Z} = 0$ and $|\solu| \equiv 1$. We find,
\[
 a(0) = \dmv{\solu}{\solu}(x,y), \quad a'(0) = 2\dmv{\solu}{\varphi}(x,y)
\]
Next, since $\varphi \cdot \solu \equiv 0$ and $(\solu)_{\R/\Z} = 0$,
\begin{equation}\label{eq:bp2cpeq}
 b'(0) = 2 c'(0) = -2 \migo{x}{y} \solu \cdot (\varphi)_{\R/\Z}
\end{equation}

Also
\[
d'(0) = -\solu(x) \cdot  (\varphi)_{\R/\Z}, \quad e'(0) = -\solu(y)\cdot (\varphi)_{\R/\Z}
\]

Thus, we find by product rule
 \[
\begin{split}
 & \left. \frac{d}{d\eps} \right |_{\eps = 0} \int_{\R/\Z}\int_{\R/\Z} F\brac{a(\eps), b(\eps) , c(\eps)}\  d(\eps)\, e(\eps) \frac{dx\ dy}{\rho(x,y)^{\frac{\alpha p}{2}}}\\
 =& 2\int_{\R/\Z}\int_{\R/\Z} G'(\dmv{\solu}{\solu}(x,y))\ \dmv{\solu}{\varphi}(x,y)\  \frac{dx\ dy}{\rho(x,y)^{\frac{\alpha p}{2}}}\\
  &+\int_{\R/\Z}\int_{\R/\Z} \left.\frac{d}{d\eps} \right |_{\eps = 0} F\brac{\dmv{\solu}{\solu}(x,y), b(\eps) , c(\eps)}  \frac{dx\ dy}{\rho(x,y)^{\frac{\alpha p}{2}}}\\
  &-\int_{\R/\Z}\int_{\R/\Z} G\brac{\dmv{\solu}{\solu}(x,y)}\ \brac{\solu(x)+\solu(y)}\cdot (\varphi)_{\R/\Z}  \frac{dx\ dy}{\rho(x,y)^{\frac{\alpha p}{2}}}\\
  =&\mathcal{Q}(\solu,\varphi) + \mathcal{R}_1(\solu,\varphi) - \mathcal{R}_2(\solu,\varphi).
\end{split}
 \]
To find the form of $\mathcal{R}_1$ we finally observe that with \eqref{eq:bp2cpeq},
\[
\begin{split}
 &\frac{d}{d\eps}\big|_{\eps =0} F(a(0),b(\eps),c(\eps)) \\
 =& -\frac{\alpha p}{2}\brac{\brac{1- \frac{1}{2} a(0)}^{-\frac{\alpha}{2}}   -1}^{\frac{p-2}{2}} \brac{\brac{1- \frac{1}{2} a(0)}^{-\frac{\alpha+2}{2}}   -1}\, c'(0)\\
 =& H(a(0)) \migo{x}{y} \solu\cdot (\varphi)_{\R/\Z}.
\end{split}
 \]
 \end{proof}

From \eqref{eq:HGest} we readily obtain 
\begin{proposition}\label{pr:loterms}
For $\solu$ satisfying \eqref{eq:ubilipschitz} we have for $i =1,2$.
\[
 |\mathcal{R}_i(\solu,\varphi)| \aleq C(\lambda)\ \|\solu\|_{L^\infty}\ \llbracket \solu \rrbracket_{W^{\beta,p}(\R/\Z)}^p\, \|\varphi\|_{L^1},
\]
\end{proposition}
Observe that by Lemma~\ref{la:lambda} we have $\lambda=\lambda(u)$ is bounded away from $2$, if $\solu=c^{-1} \sol'$ where $\sol$ has finite O'hara energy.

\subsection{Regularity theory: Left-hand side estimates}
Recall that
\[
 \mathcal{Q}(\solu,\varphi) := 2\int_{\R/\Z}\int_{\R/\Z} G'(\dmv{\solu}{\solu}(x,y))\ \dmv{\solu}{\varphi}(x,y)\ \frac{dx\ dy}{\rho(x,y)^{\alpha p /2}},
\]
where
\[
 G(s) := \brac{\brac{1- \frac{1}{2} s}^{-\frac{\alpha}{2}}  -1}^{\frac{p}{2}},
\]
For some interval $\iv$ we denote the version of $\mathcal{Q}$ restricted to $iv$ as $\mathcal{Q}_\iv$, that is
\[
 \mathcal{Q}_\iv(\solu,\varphi) := 2\int_{\iv} \int_{\iv} G'(\dmv{\solu}{\solu}(x,y))\ \dmv{\solu}{\varphi}(x,y)\ \frac{dx\ dy}{\rho(x,y)^{\alpha p /2}},
\]
and as in \cite{SchikorraCPDE14} define the potential $\Gamma_{\beta,\iv}u$ as 
\begin{equation}\label{eq:defGammabetaS}
 \Gamma_{\beta,\iv}\solu(z) := \mathcal{Q}_{\iv}(\solu,|z-\cdot|^{\beta-1}).
\end{equation}
The reason for taking such a potential is that $|z-\dot|^{\beta-1}$ is the kernel of the Riesz potential $\lapms{\beta}$,
\[
 \lapms{\beta} f(x) = \int_{\R} |z-x|^{\beta-1}\, f(z)\, dz.
\]
The Riesz potential of order $\beta$, $\lapms{\beta}$, has as an inverse the fractional Laplacian of order $\beta$, $\laps{\beta}$:
\[
 \lapms{\beta} \laps{\beta} f  = \laps{\beta} \lapms{\beta} f = f \quad \mbox{for $f \in C_c^\infty(\R)$}.
\]
The fractional Laplacian $\laps{\beta}$ is an an elliptic operator of differential order $\beta$ which can be described in different ways. For $\beta \in (0,1)$, it can be represented as an integro-differential operator
\[
 \laps{\beta} f(x) = c \int_{\R} \frac{f(y)-f(x)}{|x-y|^{1+\beta}}\, dy,
\]
or its Fourier symbol is $c |\xi|^\beta$,
\[
 \mathcal{F}\brac{\laps{\beta} f}(\xi)=c |\xi|^\beta\, \mathcal{F}(f)(\xi).
\]
The kernel $|z|^{\beta-1}$ is the fundamental solution of the operator $\laps{\beta}$ (just as $|z|^{2-n}$ is the fundamental solution of $-\Delta = \laps{2}$ in dimension $n\geq 3$), which means
\[
 \laps{\beta} \brac{|z-x|^{\beta-1}} = \delta_{x,z}.
\]
That is,
\[
 \mathcal{Q}_\iv(\solu,\cdot) = \laps{\beta} \Gamma_{\beta,\iv}\solu.
\]
or, by an integration by parts,
\[
 \mathcal{Q}_\iv(\solu,\varphi) = c\int_{\R} \Gamma_{\beta,\iv}\solu(z)\ \laps{\beta} \varphi.
\]
From now on we will assume that $\iv$ is a small interval inside $[0,1]$ and that $\solu$ is extended to $\R \backslash [-1,2]$ as an $W^{\frac{1}{p},p}$-map.

\begin{proposition}[Left-hand side estimate]\label{pr:lhsest}
Let $\iv_r$ be an interval $(x_0-r,x_0+r)$, $r \in (0,\frac{1}{2})$. If $\solu: \R/\Z \to \S^2$ satisfies \eqref{eq:ubilipschitz} then for any $\eps > 0$,
\[
\begin{split}
[\solu]_{W^{\frac{1}{p},p}(\iv_r)}^p \aleq & [\solu]_{W^{\frac{1}{p},p}(\iv_{2^Lr})} \|\chi_{\iv_{2^K r}} \Gamma_{\beta,\iv_{2^L r}} \solu\|_{L^{\frac{1}{1-\beta}}}  + \eps\, [\solu]_{W^{\frac{1}{p},p}(\iv_{2^L r})}^{p} + C_\eps\, \brac{[\solu]_{W^{\frac{1}{p},p}(\iv_{2^L r})}^{p} - [\solu]_{W^{\frac{1}{p},p}(\iv_{r})}^{p}}
\end{split}
\]
for any $L, K \in \mathbb{N}$ large enough.
\end{proposition}
\begin{proof}
In view of Proposition~\ref{pr:normalunnormalWsp}, it suffices to prove the estimate for $\llbracket \solu \rrbracket_{W^{\frac{1}{p},p}}$. Recall \eqref{eq:bracketdef} and
\[
 \llbracket \solu \rrbracket_{W^{\frac{1}{p},p}(\iv_r)}^{p} = \int_{\iv_r} \int_{\iv_r} \frac{(\dmv{\solu}{\solu}(x,y))^{\frac{p}{2}}}{\rho(x,y)^{2}}\ dx\ dy,
\]
with Lemma~\ref{la:Gpest} (using the condition \eqref{eq:ubilipschitz}),
\[
 \aeq \int_{\iv_r} \int_{\iv_r} G'(\dmv{\solu}{\solu}(x,y))\ \dmv{\solu}{\solu}(x,y) \frac{dx\ dy}{\rho(x,y)^{\frac{\alpha p}{2}}}.
\]
Let $\eta \in C_c^\infty(\iv_{2r})$, $\eta \equiv 1$ on $\iv_{r}$, with $|\nabla^k \eta| \leq C(k) r^{-k}$. Set $(\solu)_{\iv_{2r}} = (2r)^{-1} \int_{\iv_{2r}} \solu$. 
\[
 \psi(x) := \eta(x) (\solu(x)-(\solu)_{\iv_{2r}}),
\]
Observe that for $x,y \in \iv_r$,
\[
 \dmv{\solu}{\solu}(x,y) = \dmv{\psi}{\psi}(x,y).
\]
Thus, for any $L \geq 0$,
\[
  \llbracket \solu \rrbracket_{W^{\frac{1}{p},p}(\iv_r)}^{p} \aleq \int_{\iv_{2^L r}} \int_{\iv_{2^L r}} G'(\dmv{\solu}{\solu}(x,y))\ \dmv{\psi}{\psi}(x,y)\frac{ dx\ dy}{\rho(x,y)^{\frac{\alpha p}{2}}}.
\]
Now we write
\begin{align*}
 \psi(s) - \psi(t) =&(\solu(s)-\solu(t)) - (1-\eta(s)) (\solu(s)-\solu(t))\\
 & + (\eta(s)-\eta(t)) (\solu(t)-(\solu)_{\iv_{2r}}).
\end{align*}
and thus
\[
     \llbracket \solu \rrbracket_{W^{\frac{1}{p},p}(\iv_r)}^{p} \aleq I - II + III
   \]
where
  \[
  I:= \int_{\iv_{2^L r}} \int_{\iv_{2^L r}} G'(\dmv{\solu}{\solu}(x,y)) \dmv{\solu}{\psi}(x,y)\ \frac{dx\ dy}{\rho(x,y)^{\frac{\alpha p}{2}}} 
  \]
  \[
  II:=-\int_{\iv_{2^L r}} \int_{\iv_{2^L r}} G'(\dmv{\solu}{\solu}(x,y)) \frac{\rho(x,y)^{-2} \int_x^y\int_x^y \left \langle (1-\eta(s)) (\solu(s)-\solu(t)) ,\psi(s) - \psi(t) \right \rangle \ ds\ dt}{{\rho(x,y)^{\frac{\alpha p}{2}}}}\ dx\ dy
  \]
   \[III:=\int_{\iv_{2^L r}} \int_{\iv_{2^L r}} G'(\dmv{\solu}{\solu}(x,y)) \frac{\rho(x,y)^{-2} \int_x^y\int_x^y\left \langle \brac{\eta(s)-\eta(t)} \brac{\solu(t)-(\solu)_{\iv_{2r}}},\psi(s) - \psi(t) \right \rangle \ ds\ dt}{{\rho(x,y)^{\frac{\alpha p}{2}}}}\ dx\ dy\\
\]
As for $II$, we have by H\"older $\frac{p-2}{p} + \frac{2}{p} = 1$, and Jensen inequality and using Fubini/Lemma~\ref{la:Ast}-arguments 
\[
 |II| \aleq \llbracket \solu \rrbracket_{W^{\frac{1}{p},p}(\iv_{2^L r})}^{p-2}\ [\psi]_{W^{\frac{1}{p},p}(\iv_{2^L r})}\ \brac{\int_{\iv_{2^L r}} \int_{\iv_{2^L r}} (1-\eta(s))^p \frac{|\solu(t)-\solu(s)|^p}{|s-t|^{2}}\ ds\ dt}^{\frac{1}{p}}
\]
We conclude as in \cite[Proof of Lemma 3.2., Proposition D.1, Proposition D.2]{SchikorraCPDE14}. The same way we estimate $III$.
Also $I$ can be estimated as in  \cite[Proof of Lemma 3.2.]{SchikorraCPDE14}.
\end{proof}

\subsection{Estimates of the right-hand side}
Since $|\solu| = 1$ a.e., we have the estimate
\begin{equation}\label{eq:chisplit}
\|\chi_{\iv_{2^K r}} \Gamma_{\beta,\iv_{2^L r}}\solu\|_{L^{\frac{1}{1-\beta}}} \aleq  \|\chi_{\iv_{2^K r}} \solu \cdot \Gamma_{\beta,\iv_{2^L r}}\solu\|_{L^{\frac{1}{1-\beta}}} + \|\chi_{\iv_{2^K r}} \solu\wedge \Gamma_{\beta,\iv_{2^L r}}\solu\|_{L^{\frac{1}{1-\beta}}},
\end{equation}
where $\solu\wedge$ denotes the $\R^3$-cross product with $\solu$.

\begin{lemma}\label{la:ucdotGamma}
For $\beta<\frac{1}{p}$ large enough,
\[
 \|\solu \cdot \Gamma_{\beta,\iv_{2^L r}}\solu\|_{L^{\frac{1}{1-\beta}}} \aleq [\solu]_{W^{\frac{1}{p},p}(\iv_{2^{2L}r})}^{p} + \sum_{k=1}^\infty 2^{-\sigma (L+k)}[\solu]_{W^{\frac{1}{p},p}(\iv_{2^{2L+k}r})}^{p}.
\]
\end{lemma}
\begin{proof}
The fact that $|\solu| = 1$, implies
\[
\begin{split}
 &\solu(z)\cdot \int_x^y \int_x^y (\solu(s)-\solu(t)) (|s-z|^{\beta-1} - |t-z|^{\beta-1})\ ds\ dt\\
 =&-\frac{1}{2}\int_x^y \int_x^y (\solu(s)-\solu(t))(\solu(s)+\solu(t)-2\solu(z)) (|s-z|^{\beta-1} - |t-z|^{\beta-1})\ ds\ dt.
\end{split}
 \]
Thus
\[
\begin{split}
 &|\solu(z) \cdot \Gamma_{\beta,\iv_{2^L r}}\solu(z)|\\
 \aleq &\int_{\iv_{2^L r}} \int_{\iv_{2^L r}} \frac{\dmv{\solu}{\solu}(x,y)^{\frac{p-2}{2}} \brac{\rho(x,y)^{-2} \int_x^y \int_x^y \chi_S(s,t) |\solu(s)-\solu||\solu(s)+\solu-2u(z)| ||s-z|^{\beta-1} - |t-z|^{\beta-1}|\ ds\ dt}}{\rho(x,y)^{2}}\ dx\ dy\\
\aleq &\llbracket \solu \rrbracket_{W^{\frac{1}{p},p}(S)}^{p-2} \brac{\int_{\iv_{2^L r}} \int_{\iv_{2^L r}} \frac{\brac{|\solu(s)-\solu||\solu(s)+\solu-2u(z)| \left ||s-z|^{\beta-1} - |t-z|^{\beta-1} \right |}^{\frac{p}{2}}}{|s-t|^{2}}\ ds\ dt}^{\frac{2}{p}}.
\end{split}
 \]
The arguments in \cite[Lemma 6.5]{SchikorraCPDE14} then imply that if $\beta < \frac{1}{p}$ is large enough, the claim follows.
\end{proof}

\begin{lemma}\label{la:uwedgeGamma}
Let $\solu: \R/\Z \to \S^2$ solve the Euler-Lagrange equation from Lemma~\ref{la:EL}, then for any $K \in \N$ large enough,
\[
\begin{split}
  \|\chi_{\iv_{2^K r}} \solu \wedge \Gamma_{\beta,\iv_{2^{10K} r}} \solu\|_{L^{\frac{1}{1-\beta}}} \aleq& [\solu]_{W^{\frac{1}{p},p}(\iv_{2^{20K} r})}^p + 2^{-\sigma K} [\solu]_{W^{\frac{1}{p},p}(\iv_{2^{20K} r})}^{p-1}\\
  &+ [\solu]_{W^{\frac{1}{p},p}(\R)} \sum_{k=1}^\infty 2^{-\sigma(K+k)} [\solu]_{W^{\frac{1}{p},p}(\iv_{2^{20K+k}r})}^{p-1} + (2^{2K} r)\ \llbracket \solu \rrbracket_{W^{\frac{1}{p},p}(\R/\Z))}.
\end{split}
  \]
\end{lemma}
\begin{proof}
We follow the strategy in \cite[Lemma 3.5]{SchikorraCPDE14}.

Firstly, by duality we find $g \in C_c^\infty(\iv_{2^K r})$ so that $\|g\|_{L^{\frac{1}{\beta}}} \leq 1$ and
\[
 \|\chi_{\iv_{2^K r}} \solu\wedge\Gamma_{\beta,\iv_{2^L r}}\solu\|_{L^{\frac{1}{1-\beta}}} \aleq \int_{\R} \solu(z) \wedge\Gamma_{\beta,\iv_{2^L r}}\solu(z)\ g(z)\ dz = I + \sum_{k=1}^\infty II_k,
\]
with
\[
 I = \int_{\R} \laps{\beta} (\eta_{\iv_{2^{2K}}} \lapms{\beta} g) \solu(z) \wedge \Gamma_{\beta,\iv_{2^L r}}\solu(z)\ \ dz
\]
\[
 II_k = \int_{\R} \laps{2(\frac{1}{p}-\beta)} \brac{\laps{\beta} \brac{\eta_{\iv_{2^{2K+k+1}}\backslash \iv_{2^{2K+k}}} \lapms{\beta} g} \solu(z)\wedge} \cdot \lapms{2(\frac{1}{p}-\beta)}\Gamma_{\beta,\iv_{2^L r}}\solu(z)\ \ dz
\]
for the usual choice of cutoff functions on segment and annuli.

First, we treat $II$. Observe that (for $\varphi := \laps{2(\frac{1}{p}-\beta)} \brac{\laps{\beta} \brac{\eta_{\iv_{2^{2K+k+1}}\backslash \iv_{2^{2K+k}}} \lapms{\beta} g} \solu(z)\wedge}$)
\[
\begin{split}
 &\int \varphi(z)\, \lapms{2(\frac{1}{p}-\beta)}\Gamma_{\beta,\iv_{2^L r}}\solu(z)\, dz \leq \int_{\iv_{2^L r}}\int_{\iv_{2^L r}} \frac{(\dmv{\solu}{\solu}(x,y))^{\frac{p-2}{2}}}{\rho(x,y)^2} \dmv{\solu}{\lapms{\frac{2}{p}-\beta} \varphi}(x,y)\\
 \aleq & \llbracket \solu \rrbracket_{W^{\frac{1}{p},p}(\iv_{2^L r})}^{p-1} [\lapms{\frac{2}{p}-\beta} \varphi ]_{W^{\frac{1}{p},p}(\R)}\\
 \aleq & \llbracket \solu \rrbracket_{W^{\frac{1}{p},p}(\iv_{2^L r})}^{p-1} \|\varphi \|_{L^{\frac{1}{\frac{2}{p}-\beta}}}.
\end{split}
 \]
In the last step we used Sobolev embedding and that $\beta < \frac{1}{p}$.

This and the estimates of the remaining term of $II$ (see \cite[Proof of Lemma 3.5]{SchikorraCPDE14}) imply the estimate for $II$.

As for $I$, we set (some other) $\varphi :=(\eta_{\iv_{2^{2K}}} \lapms{\beta} g)$, and have
\[
 \|\laps{\beta} \varphi \|_{L^{\frac{1}{\beta}}} \leq 1.
\]
Up to a three-commutator estimate which can be treated exactly as in \cite[Proof of Lemma 3.5]{SchikorraCPDE14}, we have to deal with
\[
 I_1 := \int \laps{\beta} (\varphi \solu \wedge)\ \Gamma_{\beta,\iv_{2^{10 K}r}} \solu,
\]
\[
 I_2 := \int \varphi \laps{\beta} (\solu \wedge) \ \Gamma_{\beta,\iv_{2^{10 K}r}} \solu,
\]
Regarding $I_1$, we have.
\[
\begin{split}
 I_1 =& \int_{\iv_{2^L r}}\int_{\iv_{2^L r}} G'(\dmv{\solu}{\solu}) \frac{\dmv{\solu}{(\varphi \solu \wedge)}(x,y)}{\rho(x,y)^{\frac{\alpha p}{2}}}\ dx\ dy\\
 =& \int_{\R/\Z}\int_{\R/\Z} G'(\dmv{\solu}{\solu}) \frac{\dmv{\solu}{(\varphi \solu \wedge)}(x,y)}{\rho(x,y)^{\frac{\alpha p}{2}}}\ dx\ dy\\
 &+ 2\int_{[0,1] \backslash \iv_{2^L r}}\int_{\iv_{2^L r}} G'(\dmv{\solu}{\solu}) \frac{\dmv{\solu}{(\varphi \solu \wedge)}(x,y)}{\rho(x,y)^{\frac{\alpha p}{2}}}\ dx\ dy\\
 &+ \int_{[0,1]\backslash \iv_{2^L r}}\int_{[0,1]\backslash \iv_{2^L r}} G'(\dmv{\solu}{\solu}) \frac{\dmv{\solu}{(\varphi \solu \wedge)}(x,y)}{\rho(x,y)^{\frac{\alpha p}{2}}}\ dx\ dy.
 \end{split}
\]
The first term is the Euler-Lagrange operator $\mathcal{Q}(\solu, \varphi \solu\wedge)$. With the Euler-Lagrange equation for $\mathcal{E}^{\alpha,p}$, (note that $\solu\wedge \varphi \in T_{\solu} \S^2$), Lemma~\ref{la:EL},
\[
  |\mathcal{Q}(\solu, \sol \wedge\varphi )| = |\mathcal{R}(\solu,\varphi)| \aleq \mathcal{E}^{\alpha,p}(\solu)\ \|\varphi \|_{L^1} \aleq (2^{2K} r)\ \mathcal{E}^{\alpha,p}(\solu).
\]
As for the second term of $I_1$,
\[
\int_{\iv_{2^L r}} \int_{\R/\Z \backslash \iv_{2^L r}} \brac{\dmv{\solu}{\solu}(x,y)}^{\frac{p}{2}-2} \dmv{\solu}{(\varphi\solu \wedge) }(x,y)\ \frac{dx\ dy}{\rho(x,y)^{\frac{\alpha p}{2}}}.
\]
can be estimated  with Lemma~\ref{la:disjointest1}.
As for the third term we observe that for large enough $L \ageq K$, for $x,y \in [0,1] \backslash \iv_{2^L r}$, either
\[
 \rho(x,y) \ageq 2^{K}r,
\]
or by the support of $\varphi$,
\[
 \dmv{\solu}{(\varphi\solu \wedge) }(x,y) = 0.
\]
Thus,
\[
 \int_{[0,1]\backslash \iv_{2^L r}}\int_{[0,1]\backslash \iv_{2^L r}} G'(\dmv{\solu}{\solu}) \frac{\dmv{\solu}{(\varphi \solu \wedge)}(x,y)}{\rho(x,y)^{\frac{\alpha p}{2}}}\ dx\ dy \aleq 2^{-\sigma K} [\solu]_{W^{\frac{1}{p},p}(\iv_{2^{20K}r})}.
\]
This estimates $I_1$.

To estimate $I_2$, we use the definition of $\Gamma_{\beta,I}$, \eqref{eq:defGammabetaS}, and have
\[
 I_2 = \int_{\iv_{2^{10 K}r}}\int_{\iv_{2^{10 K}r}} G'(\dmv{\solu}{\solu}(x,y) )\ \dmv{\solu}{\lapms{\beta} (\varphi \laps{\beta} \solu\wedge)} \rangle\ \frac{dx\ dy}{\rho(x,y)^{\frac{\alpha p}{2}}}
\]
Now, as in \cite[(3.13)]{SchikorraCPDE14}, since $\solu\wedge \solu = 0$,
\[
\begin{split}
 \dmv{\solu}{(\lapms{\beta} (\varphi \laps{\beta} \solu \wedge)} = \rho(x,y)^{-2} \int_x^y \int_x^y (\solu(s)-\solu(t))\cdot \theta(s,t)\ ds\ dt,
 \end{split}
\]
where
\[
 \theta(s,t) = (\lapms{\beta} (\varphi \laps{\beta} \solu \wedge)(s) - \lapms{\beta} (\varphi \laps{\beta} \solu \wedge)(t) - \frac{1}{2} (\solu \wedge(s) - \solu \wedge(t))(\varphi(s)+\varphi(t))
\]
Now we can argue as in \cite[Lemma 6.6]{SchikorraCPDE14} to obtain the claim.

\end{proof}

\subsection{Proof of Theorem~\ref{th:eregularity}}
As usual for the regularity theory of harmonic maps, one needs to obtain a decay estimate of the localized energy (or the respective norm), from which the H\"older continuity of the solution follows from an iteration argument. The decay estimate is the following.
\begin{proposition}[Decay estimate]\label{pr:decay}
Let $\solu$ be a critical point of $\mathcal{E}^{\alpha,p}$ for $p \geq 2$, $\alpha p = 4$.

There exist $\eps > 0$ and $\tau < 1$ and $L_0 > 0$ such that whenever for $L \geq L_0$ and $[\solu]_{W^{\frac{1}{p},p}(\iv_{2^Lr})} < \eps$ then
\[
\begin{split}
[\solu]_{W^{\frac{1}{p},p}(\iv_r)}^p \aleq & \tau [\solu]_{W^{\frac{1}{p},p}(\iv_{2^L r})}^{p}
+ \sum_{k=1}^\infty 2^{-\sigma(L+k)} [\solu]_{W^{\frac{1}{p},p}(\iv_{2^{20L+k}r})}^{p}
\end{split}
\]
\end{proposition}
\begin{proof}
From Proposition~\ref{pr:lhsest} we have for any $\eps > 0$, and all $L,K$ large enough
\[
\begin{split}
[\solu]_{W^{\frac{1}{p},p}(\iv_r)}^p \aleq & [\solu]_{W^{\frac{1}{p},p}(\iv_{2^Lr})} \|\chi_{\iv_{2^K r}} \Gamma_{\beta,\iv_{2^L r}} \solu\|_{L^{\frac{1}{1-\beta}}}  + \eps [\solu]_{W^{\frac{1}{p},p}(\iv_{2^L r})}^{p} + C_\eps \brac{[\solu]_{W^{\frac{1}{p},p}(\iv_{2^L r})}^{p} - [\solu]_{W^{\frac{1}{p},p}(\iv_{r})}^{p}}.
\end{split}
\]
From \eqref{eq:chisplit}
\[
\|\chi_{\iv_{2^K r}} \Gamma_{\beta,\iv_{2^L r}}\solu\|_{L^{\frac{1}{1-\beta}}} \aleq  \|\chi_{\iv_{2^K r}} \solu \cdot \Gamma_{\iv_{2^L r}}\solu\|_{L^{\frac{1}{1-\beta}}} + \|\chi_{\iv_{2^K r}} \solu \wedge \Gamma_{\beta,\iv_{2^L r}}\solu\|_{L^{\frac{1}{1-\beta}}}.
\]
For Lemma~\ref{la:ucdotGamma} for $\beta<\frac{1}{p}$ large enough,
\[
 \|\solu \cdot \Gamma_{\beta,\iv_{2^L r}}\solu\|_{L^{\frac{1}{1-\beta}}} \aleq [\solu]_{W^{\frac{1}{p},p}(\iv_{2^{2L}r})}^{p} + \sum_{k=1}^\infty 2^{-\sigma (L+l)}[\solu]_{W^{\frac{1}{p},p}(\iv_{2^{2L+\ell}r})}^{p}.
\]
From Lemma~\ref{la:uwedgeGamma}, since the energy is finite,
\[
  \|\chi_{\iv_{2^K r}} \solu \wedge \Gamma_{\beta,\iv_{2^{10K} r}} \solu\|_{L^{\frac{1}{1-\beta}}} \aleq [\solu]_{W^{\frac{1}{p},p}(\iv_{2^{20K} r})}^p + 2^{-\sigma K} [\solu]_{W^{\frac{1}{p},p}(\iv_{2^{20K} r})}^{p-1} + \sum_{k=1}^\infty 2^{-\sigma(K+k)} [\solu]_{W^{\frac{1}{p},p}(\iv_{2^{20K+k}r})}^{p-1} + (2^{2K} r).
\]
Plugging this together, we obtain for all $L$ large enough, 
\[
\begin{split}
[\solu]_{W^{\frac{1}{p},p}(\iv_r)}^p \aleq & 
\brac{[\solu]_{W^{\frac{1}{p},p}(\iv_{2^Lr})} + \eps + 2^{-\sigma L}} [\solu]_{W^{\frac{1}{p},p}(\iv_{2^L r})}^{p}
+ C_\eps \brac{[\solu]_{W^{\frac{1}{p},p}(\iv_{2^L r})}^{p} - [\solu]_{W^{\frac{1}{p},p}(\iv_{r})}^{p}}
+ \sum_{k=1}^\infty 2^{-\sigma(L+k)} [\solu]_{W^{\frac{1}{p},p}(\iv_{2^{20L+k}r})}^{p}
\end{split}
\]
Using the hole-filling technique, adding $C_\eps [\solu]_{W^{\frac{1}{p},p}(\iv_{r})}^{p}$ to both sides, for small enough $\eps$ and large enough $L$, whenever $[\solu]_{W^{\frac{1}{p},p}(\iv_{2^Lr})} < \eps$, for some $\tau < 1$
\[
\begin{split}
[\solu]_{W^{\frac{1}{p},p}(\iv_r)}^p \aleq & \tau [\solu]_{W^{\frac{1}{p},p}(\iv_{2^L r})}^{p}
+ \sum_{k=1}^\infty 2^{-\sigma(L+k)} [\solu]_{W^{\frac{1}{p},p}(\iv_{2^{20L+k}r})}^{p}
\end{split}
\]
\end{proof}

\begin{proof}[Proof of Theorem~\ref{th:eregularity}]
Iterating the estimate from Proposition~\ref{pr:decay} on small balls, cf. \cite[Lemma A.8]{BRS16}, we find $\sigma > 0$ such that
\[
\sup_{r > 0,x \in \R /\Z} r^{-\sigma}\, [\solu]_{W^{\frac{1}{p},p}(\iv_r(x))} \aleq C(\solu).
\]
From Sobolev embedding on Morrey spaces, \cite{Adams75} we obtain that $\solu \in C^{\tilde{\sigma}}$ for any $\tilde{\sigma} < \sigma$.
\end{proof}

\renewcommand{\a}{\ensuremath{\alpha}}
\renewcommand{\b}{\ensuremath{\beta}}
\newcommand{\br}[1]{\ensuremath{\left(#1\right)}}
\renewcommand{\d}{\ensuremath{\,d}}
\newcommand{\dg}{\ensuremath{\g'}}
\newcommand{\D}[1][\g]{{#1(x)-#1(y)}}
\newcommand{\E}[1][\a,4/\a]{\ensuremath{\mathcal{O}^{#1}}}
\newcommand{\g}{\sol}
\newcommand{\sett}[2]{\ensuremath{\left\{#1\,\middle|\,#2\right\}}}
\newcommand{\sq}[1]{\ensuremath{\left[#1\right]}}
\newcommand{\rzd}{(\R/\Z,\R^{d})}

\section{Extremal cases \texorpdfstring{$\alpha=4$}{p=1} and \texorpdfstring{$\alpha=0$}{p=infty}, and no M\"obius invariance for \texorpdfstring{$p=2$}{p=2}}\label{s:distortion}
The following definition goes back to Gromov~\cite{G83}, see also
O'hara~\cite{OH92}.

\begin{definition}[distortion]
 For any curve $\g\in C^{0}\rzd$ let
 \[ \distor\g := \sup_{\substack{x,y\in\R/\Z\\x\ne y}}
 \frac{\mathcal{D}_{\g}(x,y)}{\abs{\g(x)-\g(y)}} \]
 if $\g$ is injective and $\distor\g :=\infty$ else.
\end{definition}

\begin{lemma}
 Let $\g\in C^{0,1}\rzd$ be parametrized by arc-length. Then
 \[ \liminf_{\a\searrow0}\tfrac1\a\br{\E(\g)}^{\a/4} \ge\log\distor\g. \]
\end{lemma}

\begin{proof}
 We may assume that $\g$ is injective since otherwise
 $\E(\g)$ would be infinite.
 We compute
 \begin{align*}
  &\frac1\a\br{\frac1{\abs{\D}^{\a}}-\frac1{\mathcal{D}_{\g}(x,y)^{\a}}} \\
  &=\frac1{\a\abs{\D}^{\a}}\br{1-\frac{\abs{\D}^{\a}}{\mathcal{D}_{\g}(x,y)^{\a}}} \\
  &=\frac1{\abs{\D}^{\a}}\int_{\frac{\abs{\D}}{\mathcal{D}_{\g}(x,y)}}^{\a-1}\xi^{-1}\d\xi.
 \end{align*}
 On $|x-y|\ge\delta>0$ the integral uniformly converges,
 \[ \int_{\frac{\abs{\D}}{\mathcal{D}_{\g}(x,y)}}^{1}\xi^{\a-1}\d\xi
 \xrightarrow{\a\searrow0} \int_{\frac{\abs{\D}}{\mathcal{D}_{\g}(x,y)}}^{1}\xi^{-1}\d\xi
 = \log\frac{\mathcal{D}_{\g}(x,y)}{\abs{\D}}, \]
 i.e., for given $\eps>0$ there is some $\a_{0}=\a_{0}(\delta,\eps)>0$ such
 that, for any $\a\in(0,\a_{0})$, these two quantities
 differ at most by $\eps$.
 
 By definition of distortion, for given $\eps>0$ we may choose points $x_{0},y_{0}\in\R/\Z$,
 $x_{0}\ne y_{0}$, such that
 \[ \br{\distor\g\ge}\quad \frac{\mathcal{D}_{\g}(x_{0},y_{0})}{\abs{\g(x_0)-\g(y_{0})}}
 \ge e^{-\eps}\distor\g. \]
 Moreover, we may find some $\delta=\delta(\eps)>0$
 such that the $\delta$-neighborhoods of $x_{0}$ and $y_{0}$
 are disjoint in $\R/\Z$ and
 \[ \frac{\mathcal{D}_{\g}(x,y)}{\abs{\g(x)-\g(y)}}
 \ge e^{-2\eps}\distor\g
 \qquad\text{for all }x\in B_{\delta}(x_{0}), y\in B_{\delta}(y_{0}). \]
 For any $\a\in(0,\a_{0})$ we arrive at
 \begin{align*}
  &\iint_{(\R/\Z)^{2}}\sq{\frac1\a\br{\frac1{\abs{\g(x)-\g(y)}^{\a}}-\frac1{\mathcal{D}_{\g}(x,y)^{\a}}}}^{2/\a}\d x\d y \\
  &\ge\int_{B_\delta(x_{0})}\int_{B_\delta(y_{0})}\sq{\frac1\a\br{\frac1{\abs{\g(x)-\g(y)}^{\a}}-\frac1{\mathcal{D}_{\g}(x,y)^{\a}}}}^{2/\a}\d y\d x \\
  &=\int_{B_\delta(x_{0})}\int_{B_\delta(y_{0})}\frac1{\abs{\g(x)-\g(y)}^{2}}\br{\int_{\frac{\abs{\D}}{\mathcal{D}_{\g}(x,y)}}^{1}\xi^{\a-1}\d\xi}^{2/\a}\d y\d x \\
  &\ge\int_{B_\delta(x_{0})}\int_{B_\delta(y_{0})}\frac1{\abs{\g(x)-\g(y)}^{2}}\br{\log\frac{\mathcal{D}_{\g}(x,y)}{\abs{\g(x)-\g(y)}}-\eps}^{2/\a}\d y\d x \\
  &\ge\int_{B_\delta(x_{0})}\int_{B_\delta(y_{0})}\frac1{\abs{\g(x)-\g(y)}^{2}}\br{\log\distor\g-3\eps}^{2/\a}\d y\d x \\
  &\ge\br{\log\distor\g-3\eps}^{2/\a}\frac{(2\delta)^{2}}{(\abs{x_{0}-y_{0}}_{\R/\Z}+2\delta)^{2}}.
 \end{align*}
 Thus we arrive at
 \[ \liminf_{\a\searrow0}\tfrac1\a\br{\E(\g)}^{\a/4} \ge\log\distor\g-3\eps. \]
\end{proof}

\begin{lemma}
 Let $\g\in C^{1,\beta}\rzd$ for some $\beta\in(0,1]$
 be parametrized by arc-length. Then
 \[ \limsup_{\a\searrow0}\tfrac1\a\br{\E(\g)}^{\a/4} \le\log\distor\g. \]
\end{lemma}

\begin{proof}
 We may assume that $\g$ is injective since otherwise its distortion
 would be infinite.
 As in the preceding proof we have
 \begin{align}\label{eq:distor-limsup}
  \begin{split}
  &\frac1\a\br{\frac1{\abs{\D}^{\a}}-\frac1{\mathcal{D}_{\g}(x,y)^{\a}}} \\
  &=\frac1{\a\abs{\D}^{\a}}\br{1-\frac{\abs{\D}^{\a}}{\mathcal{D}_{\g}(x,y)^{\a}}} \\
  &=\frac1{\abs{\D}^{\a}}\int_{\frac{\abs{\D}}{\mathcal{D}_{\g}(x,y)}}^{1}\xi^{\a-1}\d\xi \\
  &\le\frac1{\abs{\D}^{\a}}\int_{\frac{\abs{\D}}{\mathcal{D}_{\g}(x,y)}}^{1}\xi^{-1}\d\xi \\
  &=\frac1{\abs{\D}^{\a}} \log\frac{\mathcal{D}_{\g}(x,y)}{\abs{\D}}.
  \end{split}
 \end{align}
 We compute for some constant $C>0$ only depending on $\g$
 \begin{align*}
  1&\ge\frac{\abs{\D}^{2}}{|x-y|^{2}}\\
  &=\abs{\int_{0}^{1}\dg(x+\theta (y-x))\d\theta}^{2} \\
  &=\iint_{[0,1]^{2}}{} \left \langle\dg(x+\theta_{1} (y-x)),\dg(x+\theta_{2} (y-x))\right \rangle\d\theta_{2}\d\theta_{1} \\
  &=1-\tfrac12\iint_{[0,1]^{2}}\abs{\dg(x+\theta_{1} (y-x))-\dg(x+\theta_{2} (y-x))}^{2}\d\theta_{2}\d\theta_{1} \\
  &\geq 1-C\, |x-y|^{2\beta}.
 \end{align*}
 This is bounded below by $\frac12$ provided $|x-y|\le (2C)^{-1/2\beta}$
 (we also assume $|x-y|\le\frac12$).
 In this case, the right-hand side of~\eqref{eq:distor-limsup} is
 majorized by
 \begin{align*}
 &\frac1{\abs{\D}^{\a}} \log\frac{\mathcal{D}_{\g}(x,y)}{\abs{\D}} \\
 &\le \frac1{2|x-y|^{\a}\br{1-C|x-y|^{2\beta}}^{\a}}\log\frac{\mathcal{D}_{\g}(x,y)^{2}}{\abs{\D}^{2}} \\
 &\le \frac1{2^{1-\a}|x-y|^{\a}}\log\frac{|x-y|^{2}}{\abs{\D}^{2}} \\
 &\le \frac1{2^{1-\a}|x-y|^{\a}}\log\frac1{1-C|x-y|^{2\beta}} \\
 &\le 2^{\a-1}\frac{C|x-y|^{2\beta-\a}}{1-C|x-y|^{2\beta}} \\
 &\le 2^{\a} C|x-y|^{2\beta-\a}.
 \end{align*}
 On the other hand, if $|x-y|\ge (2C)^{-1/2\beta}$ we have
 \[ \abs\D\ge c>0\]
 by injectivity.
 Letting $\eps\in\br{0,(2C)^{-1/2\beta}}$ and $\a<4\beta$
 we arrive at
 \begin{align*}
 &\tfrac1\a\br{\E(\g)}^{\a/4} \\
 &=\br{\int_{\R/\Z}\int_{-1/2}^{1/2}\sq{\frac1\a\br{\frac1{\abs{\D}^{\a}}-\frac1{\mathcal{D}_{\g}(x,y)^{\a}}}}^{2/\a}\d w\d u}^{\a/4} \\
 &\le\br{\iint_{|x-y|\ge(2C)^{-1/2\beta}}\cdots}^{\a/4}
 + \br{\iint_{|x-y|\le(2C)^{-1/2\beta}}\cdots}^{\a/4} \\
 &\le\br{\iint_{|x-y|\ge\eps} \frac1{c^{2}}\br{\log\frac{|x-y|}{\abs\D}}^{2/\a}\d w\d u}^{\a/4} \\
 &\qquad{}+ \br{\iint_{|x-y|\le(2C)^{-1/2\beta}}4C^{2/\a}|x-y|^{4\beta/\a-2}\d w\d u}^{\a/4} \\
 &\le \frac1{c^{\a}}\log\distor\g + C\br{\frac{8\eps^{4\beta/\a-1}}{\frac{4\beta}\a-1}}^{\a/4} \\
 &\le \frac1{c^{\a}}\log\distor\g + C\br{\frac{8\a}{{4\beta-\a}}}^{\a/4}\eps^{2\beta-\a/4} \\
 &\xrightarrow{\a\searrow0}\log\distor\g+C\eps^{2\beta}.
 \end{align*}
\end{proof}

Subsuming the past two results we obtain

\begin{corollary}
 Let $\g\in C^{1,\beta}\rzd$ for some $\beta\in(0,1]$
 be parametrized by arc-length. Then
 \[ \tfrac1\a\br{\E(\g)}^{\a/4}\xrightarrow{\a\searrow0}\log\distor\g. \]
\end{corollary}

\begin{corollary}\label{cor:alpha0}
 There is some $\a_{0}\in(0,2]$ such that
 $\E$ is \emph{not} M\"obius invariant for any $\a\in(0,\a_{0})$.
\end{corollary}

The preceding statement immediately follows from

\begin{example}[Distortion is not M\"obius invariant]
 We consider the square
 \[ Q = \sett{(x,y)\in\R^{2}}{\max\br{\abs x,\abs y}\le 1} \]
 which we identify with an arbitrary (injective) arc-length parametrization.
 In order to see $\distor Q = 2$ we consider the quotient
 $\frac{\mathcal{D}_{\g}(x,y)}{\abs{\g(x)-\g(y)}}$
 for different configurations of $x,y\in\R/\Z$, $x\ne y$.
 If $x,y$ belong to the same edge of Q it clearly amounts to~$1$
 while we arrive at~$\sqrt 2$ in case of $x,y$ belonging to neighboring edges.
 For opposite edges we evaluate the quotient at points
 $(\xi,-1)$ and $(\eta,1)$, $\xi,\eta\in[-1,1]$.
 Without loss of generality we may assume $\xi+\eta\ge0$, thus
 $\frac{d_{Q}(x,y)}{\abs{Q(x)-Q(y)}} = \frac{4-\xi-\eta}{\sqrt{2^{2}+(\xi-\eta)^{2}}} \le 2$
 with equality if $\xi=\eta=0$.
 
 Now we perform an inversion on the unit circle $x^{2}+y^{2}=1$.
 It maps the edges of $Q$ onto segments of the circles
 passing to the origin and tangentially meeting the
 intersection of $Q$ with the axes, see Figure~\ref{fig:distor} (left).
 The points $\pm(1,1)$ are mapped to $\pm\br{\frac12,\frac12}$
 while the length of the two arcs joining them
 now amounts to~$\pi$.
 The distortion of the inversion of $Q$ is therefore bounded below by
 $\frac\pi{\sqrt 2} > \frac3{\sqrt 2} > 2$.
\end{example}

\begin{figure}
 \includegraphics[height=4cm]{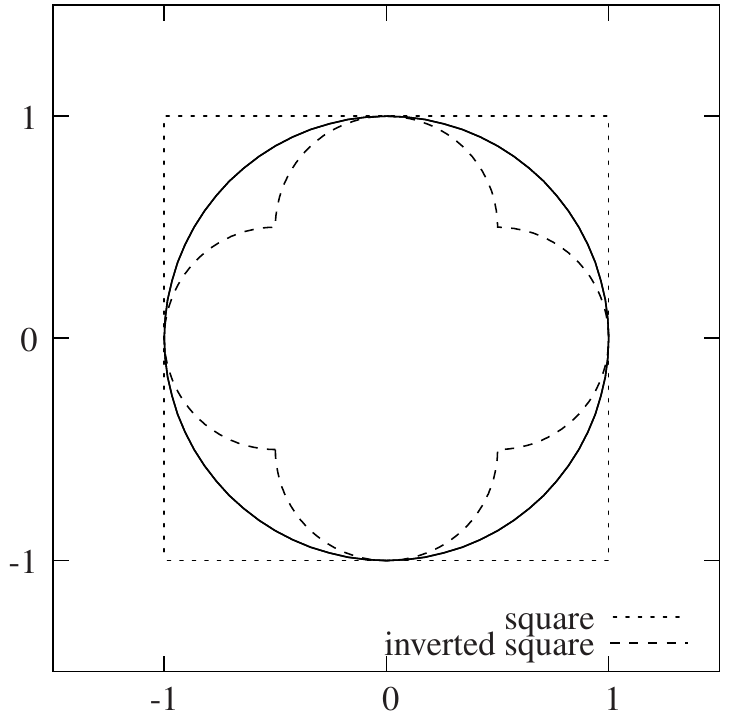}\quad
 \includegraphics[height=4cm]{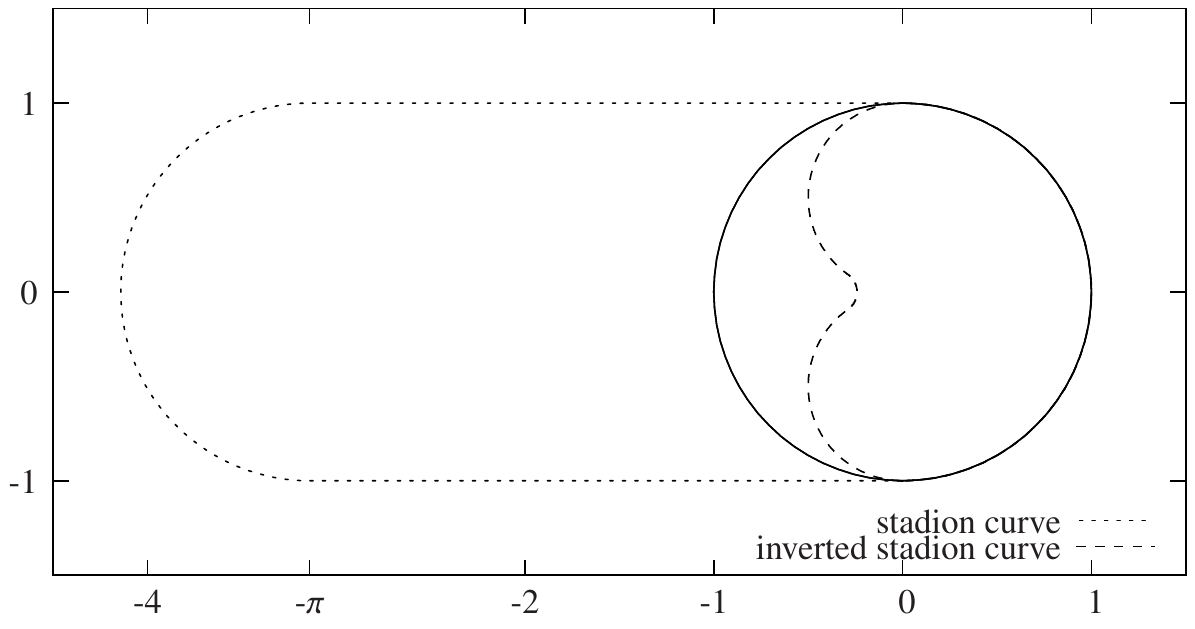}
 \caption{The square $Q$ (left) and a station curve $\sigma$ (right) inverted on the unit circle.}\label{fig:distor}
\end{figure}

Conjecture~\ref{conj} means that we have $\a_{0}=2$ in Corollary~\ref{cor:alpha0}.

Rather than providing a rigorous proof of the preceding statement by suitably
decomposing the integration domain and carefully estimating
based on techniques from~\cite{blatt-reiter-2008},
we illustrate the situation by providing a numerical experiment that supports the latter conjecture (and which we believe should be possible to be made rigorous with a careful, lengthy computation).

To this end, we consider the stadion curve $\sigma$ depicted in Figure~\ref{fig:distor} (right)
which is constructed by vertically cutting the unit circle into two half
circles and horizontally delating the left one by $\pi$.
This gives $\distor\sigma=\pi$.

Now we approximate the energy values $\tfrac1\alpha\br{\E}^{\a/4}$,
$\a\in(0,2]$, for
the (unit) circle, the stadion curve $\sigma$ and its inversion on the
unit circle.
We employ a straightforward mid-point based quadrature.
Let $\br{x_{j}}_{j=1,\dots,N}\subset\R^{2}$
denote a set of vertices which form a polygonal approximation of the curve.
As it is assumed to be closed, we set $x_{j\pm N}=x_{j}$ for all $j=1,\dots,N$.
So the approximative length $L = \sum_{\ell=1}^{N}\abs{x_{j+1}-x_{j}}$ is well-defined.
For any $j,k\in\{1,\dots,N\}$, $j\le k$, we let
$d_{j,k} = d_{k,j} = \min\br{\sum_{\ell=j}^{k-1}\abs{x_{j+1}-x_{j}},
L-\sum_{\ell=j}^{k-1}\abs{x_{j+1}-x_{j}}}$.
Furthermore, we set
$\Delta_{j} = \tfrac12\br{\abs{x_{j+1}-x_{j}}+\abs{x_{j}-x_{j-1}}}$.
In order to avoid cancellation effects for $\alpha\searrow0$
which occur when evaluating the approximative energy value
\[  \frac1\a\br{\sum_{j,k=1}^{N}\br{\frac1{\abs{x_{j+1}-x_{j}}^{\a}}-\frac1{d_{j,k}^{\a}}}^{2/\a}\Delta_{j}\Delta_{k}}^{\a/4}, \]
we use the (algebraically equal) term
\[ \log\beta \br{\sum_{j,k=1}^{N}\frac1{\abs{x_{j+1}-x_{j}}^{2}}\br{\frac{1-\br{\frac{\abs{x_{j+1}-x_{j}}}{d_{j,k}}}^{\a}}{\a\log\beta}}^{2/\a}\Delta_{j}\Delta_{k}}^{\a/4} \]
where $\beta = \max_{j,k=1,\dots,N}\frac{d_{j,k}}{\abs{x_{j}-x_{k}}}$
approximates the distortion.

The results presented in Figure~\ref{fig:energy-plot} suggest
that $\E$ is not M\"obius invariant unless $\a=2$.
In this experiment we have chosen $N=1000$.
We find that our values for the circle agree
with those obtained using the integral formula derived by
Abrams et al.~\cite[(24)]{Abrams2003} up to
a relative error of about $10^{-3}$.

\begin{figure}
 \includegraphics[scale=.8]{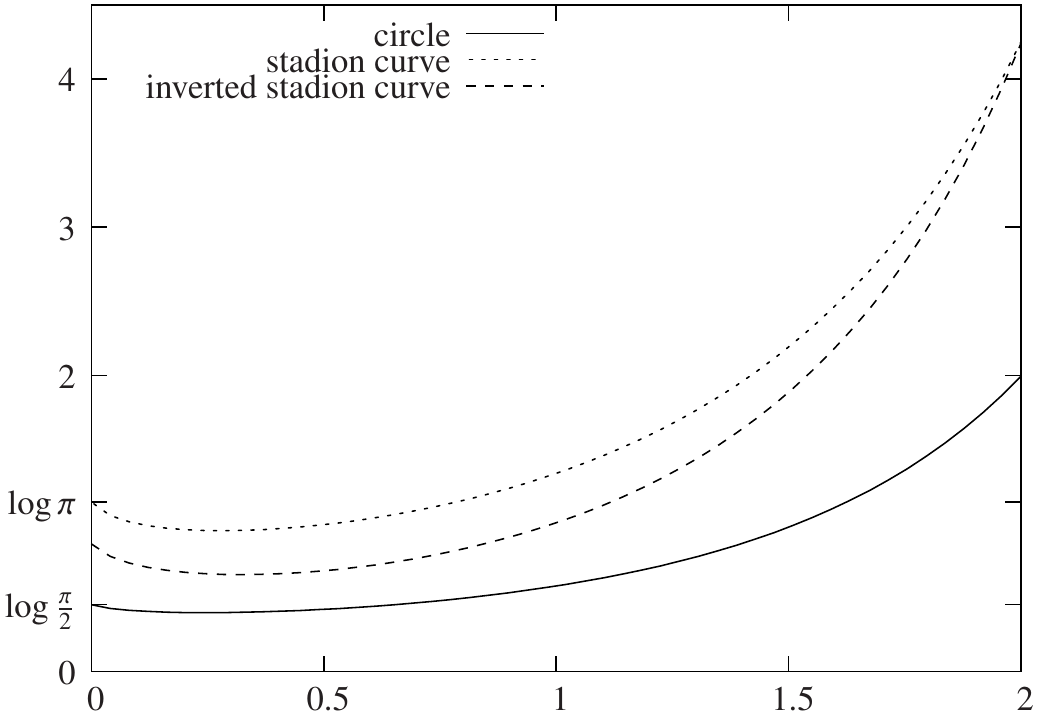}
 \caption{Plot of $\alpha\mapsto\tfrac1\alpha\br{\E}^{\a/4}$
 for the circle and the curves depicted in Figure~\ref{fig:distor} (right).}\label{fig:energy-plot}
\end{figure}

When $\alpha \to 4$ (or $p \to 1$) we converge to the total curvature functional.
\begin{proposition}
If $\sol: \R/\Z \to \R^3$ is a constant speed, bilipschitz parametrization of a knot then
\[
 \lim_{\alpha \to 4^-} \frac{4-\alpha}{4} \mathcal{O}^{\alpha,\frac{4}{\alpha}}(\sol) = c \int_{\R / \Z} |\sol''|.
\]
\end{proposition}
\begin{proof}
We have by Lemma~\ref{la:EOequal},
\[
 \mathcal{O}^{\alpha,\frac{4}{\alpha}}(\sol) = \mathcal{E}^{\alpha,\frac{4}{\alpha}} (\sol').
\]
From Lemma~\ref{la:Gpest} 
\[
 \mathcal{E}^{\alpha,\frac{4}{\alpha}} (\sol') \aeq \llbracket u \rrbracket_{W^{\frac{\alpha}{4},\frac{4}{\alpha}}(\R/\Z)}^p
\]
In view of Proposition~\ref{pr:Wspequivalence}
\[
 \llbracket \sol' \rrbracket_{W^{\frac{\alpha}{4},\frac{4}{\alpha}}(\R/\Z)}^{\frac{4}{\alpha}} \aeq [\sol']_{W^{\frac{\alpha}{4},\frac{4}{\alpha}}(\R/\Z)}^{\frac{4}{\alpha}}.
\]
From \cite{BBM01} we know that
\[
 \lim_{\alpha \to 4^-} \brac{1-\frac{4}{\alpha}} [\sol' ]_{W^{\frac{\alpha}{4},\frac{4}{\alpha}}(\R/\Z)}^{\frac{4}{\alpha}} \aeq \int_{\R/\Z}|\sol''|
\]
The claim follows.
\end{proof}

\begin{appendix}
\section{A new norm for \texorpdfstring{$W^{\beta,p}$}{W(s,p)}}
Let
\[
\begin{split}
 \dmv{u}{v}(x,y) :=& \int_0^1\int_0^1 \big (u(x+s(y-x)) - u(x+t(y-x)) \big ) \cdot \big (v(x+s(y-x)) - v(x+t(y-x)) \big )\ ds\ dt\\
 =&  \rho(x,y)^{-2} \int_x^y\int_x^y \big (u(s) - u \big ) \cdot \big (v(s) - v(t) \big )\ ds\ dt
\end{split}
 \]
We set 
\[
  \llbracket u \rrbracket_{W^{\beta,p}(S)} := \brac{\int_{\iv}\int_{\iv} \frac{\brac{\dmv{u}{u}(x,y)}^{\frac{p}{2}}}{\rho(x,y)^{1+\beta p}}\ dx\ dy}^{\frac{1}{p}}.
\]
The semi-norm $\llbracket u \rrbracket_{W^{\beta,p}(S)}$ is equivalent to the usual $W^{\beta,p}$-seminorm for $\beta$ large enough -- note that in particular for our situation where $p = \frac{1}{s}$ this equivalence holds. Here
\[
  [u ]_{W^{\beta,p}(S)} := \brac{\int_{\iv}\int_{\iv} \frac{\left | u(x)-u(y)\right |^{\frac{p}{2}}}{\rho(x,y)^{1+\beta p}}\ dx\ dy}^{\frac{1}{p}}.
\]

\begin{proposition}\label{pr:normalunnormalWsp}
For $p \in (1,\infty)$, $\beta \in (0,1)$ so that $\beta > \frac{1}{p}-\frac{1}{2}$ we have for any $u \in C_c^\infty(\R)$, for any interval $\iv \subseteq \R$,
\begin{equation}\label{eq:unnormalnormalWsp}
 \llbracket u \rrbracket_{W^{\beta,p}(S)} \aleq [u]_{W^{\beta,p}(S)}.
\end{equation}
The constant is independent of $u$ and $I$.

For $p \in (1,\infty)$, $\beta \in (0,1)$ so that $\beta > \frac{1}{p}-\frac{1}{2(p-1)}$ we have for any $u \in C_c^\infty(\R)$,
\begin{equation}\label{eq:normalunnormalWsp}
 [u]_{W^{\beta,p}(\R)} \aleq \llbracket u \rrbracket_{W^{\beta,p}(\R)}.
\end{equation}
There is a fixed number $L \in \mathbb{N}$ so that for any compact interval $\iv_\rho \subset \R$ of sidelength $\rho$ denoting with $\iv_{2^L \rho}$ the concentric interval with sidelength $2^L \rho$, we have for a fixed 
\begin{equation}\label{eq:normalunnormalWspI}
 [u]_{W^{\beta,p}(\iv_\rho)} \aleq \llbracket u \rrbracket_{W^{\beta,p}(\iv_{2^L \rho})}.
\end{equation}

\end{proposition}
To see this we need the following two Lemmata.

\begin{lemma}\label{la:Ast}
Let $A(s,t) \subset \R^2$ be the set
\begin{equation}\label{eq:Ast}
 (x,y) \in A(s,t) \Leftrightarrow \min \{x,y\} < s,t < \max \{x,y\}.
\end{equation}
Then for any $\mu > 0$, and any $s \neq t$
\begin{equation}\label{eq:intast}
 \int_{A(s,t)} \rho(x,y)^{-2-\mu} d(x,y)= \frac{2}{\mu(1+\mu)} |t-s|^{-\mu}.
\end{equation}
\end{lemma}
\begin{proof}
W.l.o.g. $s < t$. Then by symmetry we have
\[
\begin{split}
 &\int_{A(s,t)} \rho(x,y)^{-2-\mu} dy\ dx\\
 =& 2\int_{x=-\infty}^{s} \int_{y=t}^\infty \rho(x,y)^{-2-\mu} dy\ dx\\
 =& 2\int_{x=-\infty}^{s} \int_{y=t-x}^\infty |y|^{-2-\mu} dy\ dx\\
 =& \frac{2}{1+\mu}\int_{x=-\infty}^{s} (t-x)^{-1-\mu} dx\\
 =& \frac{2}{1+\mu}\int_{\rho=t-s}^{\infty} (\rho)^{-1-\mu} d\rho\\
 =& \frac{2}{\mu(1+\mu)} (t-s)^{-\mu}.
\end{split}
 \]
\end{proof}

\begin{lemma}\label{la:meanvalueintest}
Let $p >0$, $\beta \in (0,1)$, $q \in (1,\infty)$ so that  $\beta > \frac{1}{pq}-\frac{1}{q}$ then for any $u \in C_c^\infty(\R)$ and any interval $\iv \subseteq \R$,
\begin{equation}\label{eq:dualest2}
 \int_{\iv} \int_{\iv} \frac{\brac{\rho(x,y)^{-2} \int_x^y\int_x^y |u(s)-u(t)|^{q}\ ds\ dt}^{p}}{\rho(x,y)^{1+\beta p q}}\ dx\ dy \aleq [u]_{W^{\beta,pq}(S)}^{pq}.
\end{equation}
Also for $I \subseteq \R$ an interval,
\begin{equation}\label{eq:dualest3}
 \int_{\R} \int_{\R} \frac{\brac{\rho(x,y)^{-2} \int_x^y\int_x^y \chi_S(s,t)|u(s)-u(t)|^{q}\ ds\ dt}^{p}}{\rho(x,y)^{1+\beta p q}}\ dx\ dy \aleq [u]_{W^{\beta,pq}(S)}^{pq}.
\end{equation}

\end{lemma}
\begin{proof}
The \eqref{eq:dualest3} follows the same way as \eqref{eq:dualest2}. We just prove the latter one.

If $p \geq 2$, with Jensen's inequality
\[
 \frac{\brac{\rho(x,y)^{-2} \int_x^y\int_x^y |u(s)-u(t)|^{q}\ ds\ dt}^{p}}{\rho(x,y)^{1+\beta p q}} \leq \rho(x,y)^{-3-\beta p q}\ \int_x^y\int_x^y |u(s)-u(t)|^{pq}\ ds\ dt.
\]
With Fubini (note that $s,t \in I$ if $x,y \in I$), setting $A(s,t)$ as in \eqref{eq:Ast} and using \eqref{eq:intast},
\[
\begin{split}
 &\int_{\iv} \int_{\iv} \frac{\brac{\rho(x,y)^{-2} \int_x^y\int_x^y |u(s)-u(t)|^{q}\ ds\ dt}^{p}}{\rho(x,y)^{1+\beta p q}}\ dx\ dy\\
 \aleq &\int_{\iv} \int_{\iv} |u(s)-u(t)|^{pq} \int_{A(s,t)} \rho(x,y)^{-3-\beta p q}\ d(x,y)\ ds\ dt\\
 =&c\int_{\iv} \int_{\iv} |u(s)-u(t)|^{pq} |s-t|^{-1-\beta p q}\ ds\ dt\\
 =&[u]_{W^{\beta,pq}(S)}^{pq}.
\end{split}
 \]
Now assume $0 < p <2$, then Jensen's inequality fails. We use instead Poincar\`e-Sobolev inequality $W^{\alpha,pq} ((x,y)) \subset L^q((x,y))$ for $\alpha \in (0,1)$ satisfying $\alpha \geq \frac{1}{pq}-\frac{1}{q}$. More precisely we have the estimate
\[
\begin{split}
 &\rho(x,y)^{-2} \int_x^y\int_x^y \chi_S(s,t) |u(s)-u(t)|^{q}\ ds\ dt\\
 \aleq & \rho(x,y)^{\alpha q-\frac{1}{p}} \brac{\int_x^y\int_x^y \chi_S(s,t) \frac{|u(s)-u(t)|^{pq}}{|s-t|^{1+\alpha pq}}\ ds\ dt}^{\frac{1}{p}}\\
\end{split}
 \]
Observe that this estimate would have followed from directly H\"older's inequality for $p \geq 1$. We need it however for $p > 0$.

We arrive at
\[
 \frac{\brac{\rho(x,y)^{-2} \int_x^y\int_x^y |u(s)-u(t)|^{q}\ ds\ dt}^{p}}{\rho(x,y)^{1+\beta p q}} \leq \rho(x,y)^{-2-(s-\alpha)pq}\int_x^y\int_x^y \frac{|u(s)-u(t)|^{pq}}{|s-t|^{1+\alpha pq}}\ ds\ dt
\]
Our assumption on $s$ implies that we may pick $\alpha < s$. Then we may employ again Fubini, again having $A(s,t)$ as in \eqref{eq:Ast} and using \eqref{eq:intast},
\[
\begin{split}
 &\int_{\iv} \int_{\iv} \frac{\brac{\rho(x,y)^{-2} \int_x^y\int_x^y |u(s)-u(t)|^{q}\ ds\ dt}^{p}}{\rho(x,y)^{1+\beta p q}}\ dx\ dy\\
 \aleq &\int_{\iv} \int_{\iv} \frac{|u(s)-u(t)|^{pq}}{|s-t|^{1+\alpha pq}} \int_{A(s,t)} \rho(x,y)^{-2-(s-\alpha)pq}\ d(x,y)\ ds\ dt\\
 \overset{\eqref{eq:intast}}{=}&c\int_{\iv} \int_{\iv} |u(s)-u(t)|^{pq} |s-t|^{-1-\beta p q}\ ds\ dt =[u]_{W^{\beta,pq}(S)}^{pq}.
\end{split}
 \]

\end{proof}

From Lemma~\ref{la:meanvalueintest} follows Proposition~\ref{pr:normalunnormalWsp}.
\begin{proof}[Proof of Proposition~\ref{pr:normalunnormalWsp}]
The estimate \eqref{eq:unnormalnormalWsp} follows directly from Proposition~\ref{pr:normalunnormalWsp}.

For \eqref{eq:normalunnormalWsp} we employ a duality argument. We have
\[
\begin{split}
&[u ]_{W^{\beta,p}(S)}^p\\
=&\int_{\iv}\int_{\iv} \frac{|u(s)-u(t)|^{p-2} (u(s)-u(t)) \cdot (u(s)-u(t))}{|s-t|^{1+\beta p}}\ ds\ dt\\
=&c\int_{\iv}\int_{\iv} |u(s)-u(t)|^{p-2} (u(s)-u(t)) \cdot (u(s)-u(t)) \int_{A(s,t)} \rho(x,y)^{-(3+\beta p)}\ d(x,y)\ ds\ dt,
\end{split}
\]
where $A(s,t) \subset \R^2$ is defined as in \eqref{eq:Ast}.

With Fubini's theorem we then have
\[
\begin{split}
&[u ]_{W^{\beta,p}(S)}^p\\
=&c\int_{\R} \int_{\R} \frac{\rho(x,y)^{-2} \int_x^y\int_x^y \chi_{I}(s,t) |u(s)-u(t)|^{p-2} (u(s)-u(t)) \cdot (u(s)-u(t)) ds\ dt}{\rho(x,y)^{1+\beta p}}\ dx\ dy,
\end{split}
\]
Let us for now assume that $I = \R$, then simply with H\"older's inequality,
\[
\begin{split}
\leq&c\int_{\R} \int_{\R} \frac{\brac{\rho(x,y)^{-2} \int_x^y\int_x^y \chi_{I}(s,t)|u(s)-u(t)|^{(p-1)2}\ ds\ dt}^{\frac{1}{2}} \brac{\rho(x,y)^{-2} \int_x^y\int_x^y \chi_{I}(s,t)|u(s)-u(t)|^{2}\ ds\ dt}^{\frac{1}{2}}}{\rho(x,y)^{1+\beta p}}\ dx\ dy\\
\leq&c\brac{\int_{\R} \int_{\R} \frac{\brac{\rho(x,y)^{-2} \int_x^y\int_x^y |u(s)-u(t)|^{(p-1)2}\ ds\ dt}^{\frac{p'}{2}} }{\rho(x,y)^{1+\beta p}}\ dx\ dy}^{\frac{1}{p'}}\ \llbracket u \rrbracket_{W^{\beta,p}(\R)}
\end{split}
\]
Because of  $s > \frac{1}{p}-\frac{1}{2(p-1)}$, from Lemma~\ref{la:meanvalueintest}  we obtain
\[
 \brac{\int_{\R} \int_{\R} \frac{\brac{\rho(x,y)^{-2} \int_x^y\int_x^y |u(s)-u(t)|^{(p-1)2}\ ds\ dt}^{\frac{p'}{2}} }{\rho(x,y)^{1+\beta p}}\ dx\ dy}^{\frac{1}{p'}} \aleq [u]_{W^{\beta,p}(\R)}^{\frac{p}{p'}}.
\]
This proves \eqref{eq:normalunnormalWsp}.

If $\iv$ is a compact interval, we have by the above estimates
\[
\begin{split}
&[u ]_{W^{\beta,p}(\iv_\rho)}^p\\
\aleq &[u]_{W^{\beta,p}(\iv_\rho)}^{p-1}\ \llbracket u \rrbracket_{W^{\beta,p}(\iv_{2^L \rho})}\\
&+\int_{\R} \int_{\R} \chi_{\R \backslash \iv_{2^L \rho}}(x,y) \frac{\rho(x,y)^{-2} \int_x^y\int_x^y \chi_{\iv_\rho}(s,t) |u(s)-u(t)|^{p-2} (u(s)-u(t)) \cdot (u(s)-u(t)) ds\ dt}{\rho(x,y)^{1+\beta p}}\ dx\ dy\\
\end{split}
\]
Since we have disjoint support, we can be rough and estimate
\[
\begin{split}
 &\int_{\R} \int_{\R} \chi_{\R \backslash \iv_{2^L \rho}}(x,y) \frac{\rho(x,y)^{-2} \int_x^y\int_x^y \chi_{\iv_\rho}(s,t) |u(s)-u(t)|^{p-2} (u(s)-u(t)) \cdot (u(s)-u(t)) ds\ dt}{\rho(x,y)^{1+\beta p}}\ dx\ dy\\
 \aleq &\int_{\iv_\rho} \int_{\iv_\rho} |u(s)-u(t)|^p\ ds\ dt\ \int_{x < \iv_\rho} \int_{y > \iv_{2^L\rho}} \frac{\rho(x,y)^{-2} }{\rho(x,y)^{1+\beta p}}\ dx\ dy\\
 \aleq& 2^{L(-1-\beta p)} [u]_{W^{\beta,p}(\iv_\rho)}^p.
\end{split}
 \]
Having $L$ be chosen large enough, we can absorb and finished proving \eqref{eq:normalunnormalWspI}.
\end{proof}

\section{Computations}

\begin{lemma}\label{la:disjointest1}
For $\varphi \in C_c^\infty(\iv_\rho)$, $L \in \N$, $k \in \N$, we have the estimate
\[
\begin{split}
 &\brac{\int_{\iv_{2^L \rho}}\int_{\iv_{2^{L+k} \rho} \backslash \iv_{2^{L+k-1}}} \frac{1}{\rho(x,y)^{2}} \brac{\rho(x,y)^{-2} \int_x^y \int_x^y |\solu(s)-\solu(t)||\varphi(s) v(s) - \varphi(t) v(t)|\ ds\ dt}^{\frac{p}{2}}\ dx\ dy}^{\frac{2}{p}}\\
 \aleq& 2^{-s(L+k)} \ \|v\|_\infty\ [\solu]_{W^{\frac{1}{p},p}(\iv_{2^{L+k}\rho })}\ \|\laps{\frac{1}{p}}\varphi\|_{L^{p}}\\
\end{split}.
\]
\end{lemma}
\begin{proof}
W.l.o.g. $\iv_\rho$ is centered around $0$. 

Note that the support of $\varphi$ tells us that whenever $x,y > \rho$ or $x,y < -\rho$,
\[
\frac{1}{\rho(x,y)^{2}} \brac{\rho(x,y)^{-2} \int_x^y \int_x^y |\solu(s)-\solu(t)||\varphi(s) v(s) - \varphi(t) v(t)|\ ds\ dt}^{\frac{p}{2}} \equiv 0.
\]
Thus, it suffices to estimate 
\[
\begin{split}
 &\int_{-2^L \rho}^{\rho}\int_{2^{L+k-1} \rho}^{2^{L+k}\rho} \frac{1}{\rho(x,y)^{1+\frac{1}{p}}} \brac{\rho(x,y)^{-2} \int_x^y \int_x^y |\solu(s)-\solu(t)||\varphi(s) v(s) - \varphi(t) v(t)|\ ds\ dt}^{\frac{p}{2}}\ dx\ dy\\
 \aleq& (2^{L+k} \rho)^{-2-p} (2^L \rho) (2^{L+k} \rho) \brac{\int_{\iv_{2^{L+k}\rho}} \int_{\iv_{2^{L+k}\rho}} |\solu(s)-\solu(t)||\varphi(s) v(s) - \varphi(t) v(t)|\ ds\ dt}^{\frac{p}{2}}\\
 \aleq& (2^{L+k} \rho)^{-1-p} (2^L \rho)\ \|v\|_\infty^{\frac{p}{2}} (2^{L+k} \rho)^{p-2} \int_{\iv_{2^{L+k}\rho}} \int_{\iv_{2^{L+k}\rho}} |\solu(s)-\solu(t)|^{\frac{p}{2}}\ |\varphi(s) - \varphi(t) |^{\frac{p}{2}}\ ds\ dt\\
 &+ (2^{L+k} \rho)^{-1-p} (2^L \rho)  (2^{L+k} \rho)^{p-2} \int_{\iv_{2^{L+k}\rho}} \int_{\iv_{2^{L+k}\rho}} |\solu(s)-\solu(t)|^{\frac{p}{2}}\ |\varphi(t)|^{\frac{p}{2}} |v(t)-v(s)|^{\frac{p}{2}}\ ds\ dt\\
 =& (2^{L+k} \rho)^{-3} (2^L \rho)\ \|v\|_\infty^{\frac{p}{2}} \int_{\iv_{2^{L+k}\rho}} \int_{\iv_{2^{L+k}\rho}} |\solu(s)-\solu(t)|^{\frac{p}{2}}\ |\varphi(s) - \varphi(t) |^{\frac{p}{2}}\ ds\ dt\\
 &+ (2^{L+k} \rho)^{-3} (2^L \rho)   \int_{\iv_{2^{L+k}\rho}} \int_{\iv_{2^{L+k}\rho}} |\solu(s)-\solu(t)|^{\frac{p}{2}}\ |\varphi(t)|^{\frac{p}{2}} |v(t)-v(s)|^{\frac{p}{2}}\ ds\ dt\\
 \aleq& (2^{L+k} \rho)^{-3} (2^L \rho)\ \|v\|_\infty^{\frac{p}{2}} \brac{\int_{\iv_{2^{L+k}\rho}} \int_{\iv_{2^{L+k}\rho}} |\solu(s)-\solu(t)|^{p}\ ds\ dt}^{\frac{1}{2}}\ \brac{2^{L+k} \rho\ \|\varphi\|^p_{L^p}}^{\frac{1}{2}}\\
 \aleq& (2^{L+k} \rho)^{-\frac{5}{2}} (2^L \rho)\ \|v\|_\infty^{\frac{p}{2}}\ \brac{(2^{L+k} \rho )^{1+1} [\solu]_{W^{\frac{1}{p},p}(\iv_{2^{L+k}\rho }) }^p}^{\frac{1}{2}}\ \brac{\rho^{1} \|\laps{\frac{1}{p}}\varphi\|_{L^{p}}^p }^{\frac{1}{2}}\\
  =& (2^{L+k} )^{-\frac{3}{2}} (2^L )\ \|v\|_\infty^{\frac{p}{2}}\ \brac{[\solu]_{W^{\frac{1}{p},p}(\iv_{2^{L+k}\rho }) }^p}^{\frac{1}{2}}\ \brac{ \|\laps{\frac{1}{p}}\varphi\|_{L^{p}}^p }^{\frac{1}{2}}\\
 \end{split}
 \]

\end{proof}

\end{appendix}

\bibliographystyle{abbrv}%
\bibliography{bib}%

\end{document}